\documentclass{jcmlatexr}
\def\JCMvol{xx}
\def\JCMno{x}
\def\JCMyear{200x}
\def\JCMreceived{xxx}
\def\JCMrevised{xxx}
\def\JCMaccepted{xxx}

\usepackage{algorithmic,caption,booktabs}
\usepackage[ruled]{algorithm}
\usepackage{longtable}
\usepackage{multirow}
\usepackage{epsfig}
\usepackage{epstopdf}
\usepackage{setspace}
\usepackage{caption}
\usepackage[numbers,sort&compress]{natbib}
\usepackage{amsmath}
\allowdisplaybreaks[4]
\begin{document}
\linespread{1.2}
\markboth{Y.G.~PEI, Q.H.~GAO, J.Y. Wang, S.F.~Song AND D.T.~ZHU}{A Line Search Filter Sequential Adaptive Cubic Regularisation Algorithm}

\title{A LINE SEARCH FILTER SEQUENTIAL ADAPTIVE CUBIC REGULARISATION ALGORITHM FOR NONLINEARLY CONSTRAINED OPTIMIZATION}

\author{
Yonggang Pei
\thanks{School of Mathematics and Statistics, Henan Normal university, Xinxiang 453007, China \\ Email: peiyg@163.com}
\and
Jingyi Wang
\thanks{School of Mathematics and Statistics, Henan Normal university, Xinxiang 453007, China \\ Email: 15993075691@163.com}
\and
Shaofang Song
\thanks{School of Mathematics and Statistics, Southwest University, Chongqing 400715, China \\ Email: ssf731713505@163.com}
\and
Qinghui Gao
\thanks{Faculty of Education, School of Mathematics and Statistics, Henan Normal university, Xinxiang 453007, China \\ Email: gaoqinghui@htu.edu.cn}
\and
Detong Zhu
\thanks{Mathematics and Science College, Shanghai Normal University, Shanghai 200234, China \\ Email: dtzhu@shnu.edu.cn}
}

\maketitle

\begin{abstract}
In this paper, a sequential adaptive regularization algorithm using cubics (ARC) is presented to solve nonlinear equality constrained optimization.  It is motivated by the idea of handling constraints in sequential quadratic programming methods. In each iteration, we decompose the new step into the sum of the normal step and the tangential step by using composite step approaches. Using a projective matrix, we transform the constrained ARC subproblem into a standard ARC subproblem which generates the tangential step. After the new step is computed, we employ line search filter techniques to generate the next iteration point. Line search filter techniques enable the algorithm to avoid the difficulty of choosing an appropriate penalty parameter in merit functions and the possibility of solving ARC subproblem many times in one iteration in ARC framework. Global convergence is analyzed under some mild assumptions. Preliminary numerical results and comparison are reported.
\end{abstract}



\section{Introduction}
\label{sec:into}
Optimization algorithms play an important role in many fields, such as artificial intelligence, machine learning, signal processing, modeling design, transportation analysis, industry, structural engineering, economics, etc \cite{Curtis2020IEEE,Conn}.
Sequential quadratic programming (SQP) methods, which generate steps by solving quadratic subproblems, are among the most effective methods for nonlinearly constrained optimization. SQP methods have shown their strength when solving problems with significant nonlinearities in the constraints \cite{1999Nonlinear,Nocedal2006,sunyuan2006book,NeculaiAndrei2022} since it was first proposed by Wilson in \cite{wilson1963simplicial}.

SQP methods are often embedded in two fundamental strategies, line search and trust-region, to solve nonlinearly constrained optimization. Recently, a third alternative strategy, the adaptive regularization method using cubics (ARC), is presented. ARC was proposed by Cartis et~al. \cite{Cartis2011A} for solving unconstrained optimization. It can be viewed as an adaptive version of the cubic regularization of the Newton's method which was proposed by Griewank \cite{griewank1981} and its global convergence rates were first established by Nesterov and Polyak \cite{Nesterov2006}. Benson and Shanno \cite{BensonShanno2014COAP} also described the development of cubic regularization methods. Recently, Bellavia, Gurioli, Morini, and Toint proposed an adaptive regularization method for nonconvex optimization using inexact function values and randomly perturbed derivatives \cite{JC2022inexact}. ARC has shown its attractive convergence properties and promising numerical experiments performance \cite{Cartis2011B} for solving unconstrained optimization. It can also be viewed as a non-standard trust region method while it offers an easy way to avoid difficulties resulting from the incompatibility of the intersection of linearized constraints with trust-region bounds in constrained optimization.

In this paper, we consider how to extend ARC to solve nonlinearly constrained optimization by referring to the idea of SQP methods, and propose a penalty-free sequential adaptive cubic regularization algorithm. In each iteration, two problems should be addressed. One is the computation of a new step. The other is the decision of a new iteration point by using the new step. To obtain a new step, we need to deal with an ARC subproblem with linearized constraints. Composite step approaches are utilized to compute the new step which is decomposed into the sum of the normal step and the tangential step \cite{Conn}. The normal step is computed firstly and aims to reduce the constraint violation degree while it satisfies the linearized constraints. The tangential step is used to present sufficient decrease of the model. Using a projective matrix, we can transform the constrained ARC subproblem into a standard ARC subproblem which generates the tangential step. This projective matrix allows us to avoid the computational difficulties caused by poor choice of basis of null space of the Jacobian of the constraints in reduced Hessian methods.

After the new step is computed, we employ line search filter techniques in \cite{2005Line} to generate the next iteration point. Line search filter techniques can help us avoid the difficulty of choosing an appropriate penalty parameter in merit functions and the possibility of solving ARC subproblem many times in one iteration in ARC framework. So, the proposed algorithm is also called as line search filter sequential adaptive regularization algorithm using cubics (LsFSARC). The adaptive parameter in ARC is adjusted by the ratio of the reduction of the objective function to the reduction of the model. This is different from the standard ARC methods, where the reduction ratio is used to both update the adaptive parameter and decide the acceptance of the trial step. Global convergence is analyzed under some mild assumptions.

The following part of this paper is developed as follows. In Section \ref{sec2}, the computation of search direction is described and the filter sequential ARC algorithm combining line search for equality constrained optimization is developed as shown in Algorithm 2.1. The analysis of the global convergence to first-order critical point is presented in Section \ref{sec3}. Preliminary numerical results comparison are presented in Section \ref{sec4} and conclusion is reported in Section \ref{sec5}.

\emph{\bf Notation:}
Throughout this paper, we denote the transpose of a vector $v$ by $v^T$ and denote the transpose of  a matrix $A$ by  $A^T$.  Norms $\|\cdot\|$ denote the Euclidean norm and its compatible matrix norm. $ |\mathcal{A}|$ denotes the number of elements in a set $ \mathcal{A}$. Finally, we denote by $O(t_k)$ a sequence $\{v_k\}$ satisfying $\|v_k\|\leq \beta t_k$ for a constant $\beta>0$ independent of $k$.
%
%
%
%
%
%


\section{Line search filter sequential ARC algorithm}\label{sec2}
\label{sec:manu}
In this section, we focus on constructing an ARC algorithm combining line search filter technique to solve the nonlinear equality constrained optimization problem
\begin{subequations}\label{aaa}
	\begin{eqnarray}
		{\underset {x\in\mathbb{R}^n}{\mbox{minimize}}}   &\ \ \  f(x) \ \\
		\mbox{subject to} &\ \ \  c(x)= 0,
	\end{eqnarray}
\end{subequations}
where the objective function $f: \mathbb{R}^n \rightarrow \mathbb{R}$ and the equality constraints $c: \mathbb{R}^n \rightarrow \mathbb{R}^m$ are sufficiently smooth functions with $m\leq n$.
First, the process of computing the trial step is presented. Then, we give the acceptance mechanism for the trial step. The whole algorithm is reported in the end of this section.

Consider the constrained optimization problem \eqref{aaa}. Let $A(x)$  denote the Jacobian matrix of $c(x)$, namely,
	\begin{equation*}
		A(x)^T=[\nabla c_1(x), \nabla c_2(x), \ldots, \nabla c_m(x)].
	\end{equation*}
Assume that $A(x)$ has full row rank. We can define a projective matrix
	\begin{eqnarray}\label{eq2.1}
		P(x):=I-A(x)^T(A(x)A(x)^T)^{-1}A(x),
	\end{eqnarray}
which is a projection onto the null space of $A(x)$. This projective matrix allows us to avoid the computational difficulties caused by poor choice of basis of null space of the Jacobian of the constraints in reduced Hessian methods.

Meanwhile, we use the same definition of Lagrange function in \cite{P2018}
	\begin{eqnarray}\label{eq2.2}
		\ell(x):=f(x)-\lambda(x)^Tc(x),
	\end{eqnarray}
where $\lambda(x)$ is a projective version of the multiplier vector
	\begin{eqnarray}\label{eq2.3}
		\lambda(x):=(A(x)A(x)^T)^{-1}A(x)g(x)\in\mathbb{R}^m
	\end{eqnarray}
with $g(x):= \nabla f(x)$ denoting the gradient of the objective function $ f(x)$.

Hence, the KKT conditions can be expressed as
\begin{eqnarray*}
g(x)+A(x)y=0,~~~c(x)=0
\end{eqnarray*}
for some $y\in \mathbb{R}^m$. Equivalently, the KKT conditions can be written as
\begin{eqnarray}\label{KKT}
P(x)g(x)=0,~~~c(x)=0.
\end{eqnarray}

To obtain search directions, we need to deal with the following subproblem which is similar to SQP methods in iteration $k$,
	\begin{subequations}\label{model}
		\begin{align}
			{\underset {d\in\mathbb{R}^n}{\mbox{minimize}}} &\ \ \   f_k+g_k^Td+\frac{1}{2}d^TH_kd+\frac{1}{3}\sigma_k\|d\|^3 \\
			\mbox{subject to} &\ \ \  A_kd+c(x_k)=0,
		\end{align}
	\end{subequations}
where $f_k:=f(x_k)$, $g_k:=\nabla f(x_k)$, $A_k:=A(x_k)$, $c(x_k):=c(x_k)$, $H_k$ denotes the Hessian of Lagrange function  $\nabla_{xx}\ell(x_k)$  or its approximation and  $\sigma_k\in\mathbb{R}^+$ is an adaptive parameter in ARC. In addition, we assume that $A_k$ has full row rank for all $k$.

Instead of solving subproblem \eqref{model} directly, we decompose the overall step via composite methods as follows.
	\begin{equation*}
		d_k=n_k+t_k,
	\end{equation*}
where $n_k$ is called as a normal step which is used to satisfy feasibility condition, and $t_k$ is called as a tangential step for ensuring sufficient decrease of the function's model.

First, we can compute $n_k$ by
	\begin{eqnarray}\label{qjn}
		n_k=-A_k^T(A_kA_k^T)^{-1}c(x_k).
	\end{eqnarray}
Moreover, to ensure sufficient reduction in model function, we also require that the following condition
	\begin{eqnarray}\label{nkjie}
		\|n_k\|\leq\beta_1{\rm min}\bigg\{1,\frac{\beta_2}{\sqrt{\sigma_k}^{\beta_3}}\bigg\}\frac{1}{\sqrt{\sigma_k}},
	\end{eqnarray}
where fixed constants
$\beta_1,\beta_2>0$ and $\beta_3\in(0,1)$.

However, \eqref{nkjie} may not hold. So we distinguish two cases depending on that whether \eqref{nkjie} holds.
First, we consider the case \eqref{nkjie} holds.
The tangential step $t_k$ is computed as
	\begin{eqnarray}\label{tn}
		t_k=P_ku_k,
	\end{eqnarray}
where $P_k:=P(x_k)$ and $u_k$ is the solution(or its approximation) of the following problem
	\begin{equation}\label{qjt}
		{\mbox{minimize}}\ \ \  f(x_k)+(P_kg_k)^Tu+\frac{1}{2}u^T(P_kH_kP_k)u+\frac{1}{3}\sigma_k\|P_ku\|^3.
	\end{equation}
	The above problem is constructed by using reduced Hessian methods.
	
	After computing $u_k$, from \eqref{tn} and the definition of $P_k$, we can define
		\begin{eqnarray}\label{99}
			m_k^{t}(t_k):=f(x_k)+g_k^Tt_k+\frac{1}{2}t_k^TH_kt_k+\frac{1}{3}\sigma_k\|t_k\|^3.
		\end{eqnarray}
	Next, we discuss the mechanism of the acceptance of the trial point $x_k(\alpha_{k,l})$.
	
	For the purpose of obtaining the next iteration $x_{k+1}$, we need to determine a step size $\alpha_k$ so that $x_{k+1}=x_k+\alpha_kd_k$. To this end, we use a backtracking line search procedure combining filter method where a decreasing sequence of step sizes $\alpha_{k,l}\in(0,1](l=0,1,2,...)$ with ${\rm lim}_{l\to\infty}\alpha_{k,l}=0$ is tried until some acceptance rules are satisfied and the trial point is accepted by the current filter. For brevity, the trial point $x_k(\alpha_{k,l})$ is denoted by
		\begin{eqnarray*}
			x_k(\alpha_{k,l}):=x_k+\alpha_{k,l}d_k.
		\end{eqnarray*}
	These acceptance rules are reported in detail as follows.
	

Define the constraint violation measure $$ h(x):=\|c(x)\|.$$ We use the same definition of a filter in \cite{2005Line} where the filter is defined as a set $\mathcal{F}_k\subseteq [0,\infty)\times \mathbb{R}$ containing all prohibited $(h(x_j),\ell(x_j))$ pairs in iteration $k$.
%
At the beginning, we can set $\mathcal{F}_0=\{(h,\ell)\in\mathbb{R}^2:h>h(x_0)\}$.

A trial point $x_k(\alpha_{k,l})$ can be accepted only if it provides satisfying improvement of the infeasibility measure $h(x)$
	or $\ell(x)$, i.e.,
		\begin{eqnarray}\label{2.17}
			h(x_k(\alpha_{k,l}))\leq (1-\gamma_h)h(x_j)~~{\rm or~~}\ell(x_k(\alpha_{k,l}))\leq \ell(x_j)-\gamma_\ell h(x_j)
		\end{eqnarray}
	holds for all $(h(x_j),\ell(x_j))$ in $\mathcal{F}_k$ with fixed constants $\gamma_h,\gamma_\ell\in(0,1)$.

A trial point $x_k(\alpha_{k,l})$ is called to be acceptable to the filter $\mathcal{F}_k$  if
$$(h(x_k(\alpha_{k,l})),\ell(x_k(\alpha_{k,l})))\notin\mathcal{F}_k.$$
	
This criterion indicates that, provided that $\{\ell(x_k)\}$ is monotonically decreasing and bounded below, a sequence $\{x_k\}$ is forced towards feasibility. However, this type of sequence $\{x_k\}$ could still be accepted even if it converges to a nonoptimal point. In order to prevent this from happening, define a model
		\begin{eqnarray}\label{eq2.12}
			m_{k}(\alpha):=\alpha g_k^Tt_k+\frac{1}{2}\alpha^2t_k^TH_kt_k+\frac{1}{3}\alpha^3\sigma_k\|t_k\|^3-\alpha(\nabla\lambda_k^Td_k)^Tc(x_k),
		\end{eqnarray}
where $\nabla\lambda_k:=\nabla\lambda(x_k)$ denotes the gradient of $\lambda(x)$ at $x_k$. The following condition
		\begin{eqnarray}\label{lang}
			\ell(x_k(\alpha_{k,l}))\leq \ell(x_k)+\mu m_k(\alpha_{k,l})
		\end{eqnarray}
 is employed to be the acceptance criterion whenever the following switching conditions
		\begin{eqnarray}\label{999}
			m_k(\alpha_{k,l})<0~~~{\rm and}~~~(-m_k(\alpha_{k,l}))^{\omega}(\alpha_{k,l}\sqrt{\sigma_k})^{\omega-1}>\kappa_h(h(x_k))^{\varsigma}
		\end{eqnarray}
	hold for the current trial step size $\alpha_{k,l}$, where $0<\mu<1$, $\kappa_h>0,~\omega \geq 1$ and $\varsigma>2$ are fixed constants.

In order to tackle the situation where no acceptable step can be found and the feasibility restoration procedure has to be started, we set a threshold
\begin{eqnarray}\label{alphamin}
			\alpha_{k}^{\min}:=
			\begin{cases}
				\mu_{\alpha}\min\left\{\gamma_h,\frac{\gamma_lh(x_k)}{-g_k^Tt_k+(\nabla\lambda_k^Td_k)^Tc(x_k)},\frac{\kappa_h[h(x_k)]^{\phi}
					\sigma_{k_j}^{1-\tau}}{[-g_k^Tt_k+(\nabla\lambda_k^Td_k)^Tc(x_k)]^{\tau}} \right\} &~{\rm if}~  \delta_k>0,\\
				\mu_{\alpha}\gamma_h &~{\rm otherwise},\\				
			\end{cases}
\end{eqnarray}
with $\delta_k:=-g_k^Tt_k+(\nabla\lambda_k^Td_k)^Tc(x_k)$ and  fixed constants $\mu_{\alpha}\in (0,1]$, $\phi>2,\tau\geq1$.
The algorithm goes to feasibility restoration procedure if $\alpha_{k,l}<\alpha_{k}^{\min}$.	

Next we consider the case \eqref{nkjie} does not hold. In this case, we use the same strategy as in \cite{2005Line}. That is, the algorithm relies on the feasibility restoration procedure, whose purpose is to generate a new iterate $x_{k+1}=x_k+r_k$ which is acceptable for $\mathcal{F}_k$ and satisfies \eqref{nkjie}, where $r_k$ is a solution of the following problem
		\begin{eqnarray}\label{2.16}
			\underset{x\in\mathbb{R}^{n}}{\rm min}~h^2(x)
		\end{eqnarray}
 from $x_k$.
For convenience, we denote the set
$\mathcal{A}:=\{k~\vert ~x_k~{\rm is~added~to~the~filter}\}.$
One can see that $\mathcal{F}_k\subsetneqq\mathcal{F}_{k+1}\iff k\in\mathcal{A}$. Let $\mathcal{A}_{\rm inc}$ be the set of all indices of those iterations in which the feasibility of restoration procedure is invoked when \eqref{nkjie} does not hold.

Now, we are ready to state the linear search filter sequential adaptive regularisation algorithm with cubics for solving problem \eqref{aaa} as shown in Algorithm 2.1.	
\begin{algorithm}\label{algorithm1}
		\caption{LsFSARC for nonlinear equality constrained optimization.}
		\begin{algorithmic}
			\item [Step 0.] Initialization.\\
			(i)~Given starting point $x_0$, an initial $\sigma_{0}>0$ such that $\sigma_{\rm min}\leq\sigma_0$, an initial symmetric matrix $H_0$.\\
			(ii)~Set constants $1<\gamma_1\leq\gamma_2$, $0<\eta_1<\eta_2<1$, $\beta_1\in(0,1]$, $\beta_2>0$, $\beta_3$, $\beta$, $\kappa_h$, $\gamma_h$, $\gamma_\ell\in(0,1)$, $\varsigma>2$, $\omega\geq1$, $0<\omega_1\leq\omega_2<1$.\\
			(iii)~Set the filter $\mathcal{F}_0=\{(h,\ell)\in\mathbb{R}^2:h\geq h_{\rm max}>h(x_0)\}$
			and the iteration counter $k=0$.
			\item [Step 1.] Compute $f_k,g_k,h(x_k),A_k,\lambda_k:=\lambda(x_k),H_k,P_k$. \
			\item [Step 2.] Stop if $x_k$ is a stationary point of optimization problem \eqref{aaa}, i.e., if it satisfies the KKT conditions \eqref{KKT}.
			\item [Step 3.] Compute $n_k$ by \eqref{qjn}. If \eqref{nkjie} holds, compute $t_k$ by \eqref{qjt} and set $d_k=n_k+t_k$. Otherwise, go to step 13.
			\item [Step 4.] Compute $\alpha_k^{\rm min}$ with $\alpha_k^{\rm min}$ defined by \eqref{alphamin}. Set $\alpha_{k,0}=1$ and $l=0$.
			\item [Step 5.] If $\alpha_{k,l}<\alpha_k^{\rm min}$, go to step 13. Otherwise, compute the new trial point $x_k(\alpha_{k,l})=x_k+\alpha_{k,l}d_k$.
			\item [Step 6.] If $(h(x_k(\alpha_{k,l})),\ell(x_k(\alpha_{k,l})))\in \mathcal{F}_k$, reject $\alpha_{k,l}$ and go to step 10.
			\item[Step 7.] If \eqref{999} holds, go to step 8. Otherwise go to step 9.
			\item [Step 8.] If \eqref{lang} holds, set $\alpha_{k}=\alpha_{k,l}$, $x_{k+1}=x_k(\alpha_k)=x_k+\alpha_kd_k$ and go to step 11. Otherwise, go to step 10.
			\item [Step 9.] If \eqref{2.17} holds, set $\alpha_k=\alpha_{k,l}$, $x_{k+1}=x_k(\alpha_k)=x_k+\alpha_kd_k$, add $x_k$ to the filter $\mathcal{F}_k$  and go to step 11. Otherwise, go to step 10.
			\item [Step 10.] Choose $\alpha_{k,l+1}\in[\omega_1\alpha_{k,l},\omega_{2}\alpha_{k,l}]$, set $l=l+1$ and go back to step 5.
			\item [Step 11.] If $m_k(\alpha_k)<0$, compute
			\begin{eqnarray}\label{rhok}
				\rho_k:=\frac{\ell(x_k+\alpha_kd_k)-\ell(x_k)}{m_k(\alpha_k)}
			\end{eqnarray}
			and set
			\begin{equation*}
				\sigma_{k+1}\in
				\begin{cases}
					(\gamma_1\sigma_k,\gamma_2\sigma_k] & \text{if $\rho_k<\eta_1$},\\
					(\sigma_k,\gamma_1\sigma_k] & \text{if $\rho_k\in[\eta_1,~\eta_2)$},\\
					(0,\sigma_k] & \text{if $\rho_k\geq\eta_2$}.\\
				\end{cases}
			\end{equation*}
			Otherwise, set $\sigma_{k+1}\in (\gamma_1\sigma_k,\gamma_2\sigma_k]$.
			\item [Step 12.] Set $k=k+1$ and go to step 1.
			\item [Step 13.] Feasibility restoration procedure.\\
			13.1~~Compute a new iterate point $x_{k+1}$ by decreasing $h(x)$ for which $x_{k+1}$ satisfies both \eqref{2.17} and $(h(x_{k+1}),\ell(x_{k+1}))\notin\mathcal{F}_k$.\\
			13.2~~Determine $\sigma_{k+1}$ and go to step 12.\\
		\end{algorithmic}
	\end{algorithm}
\section{Global convergence}\label{sec3}
$Assumptions$ G. Let $\{x_k\}$ be the sequence produced by Algorithm  2.1,  where restoration iteration terminates successfully and the algorithm does not stop at a KKT point.\\
(G1)~~The iterations $\{x_k\}\subset X$, where $X$ is a closed, bounded domain with $X\subset\mathbb{R}^n$.\\
(G2)~~$f(x)$ and $c(x)$ are differentiable on $X$, and $\nabla f(x)$ and $\nabla c(x)$ are Lipschitz-continuous over $X$.\\
(G3)~~There exists a constant $M_{H}>0$ so that $\|H_k\| \leq M_{H}$ for all $k$.\\
(G4)~~$H_k$ is semipositive definite on the null space of the Jocabian $A_k$ for each $k$.\\
(G5)~~There exist constants $\delta_h,\kappa_n>0$ such that if $h(x_k)\leq\delta_h$,
	\begin{eqnarray}\label{3.1}
		k\notin\mathcal{A}_{\rm inc}~~~{\rm and~~~}\|n_k\|\leq\kappa_nh(x_k).
	\end{eqnarray}
(G6)~~There exists a constant $M_A>0$ such that
	\begin{eqnarray*}
		\varrho_{\min}(A_k)\geq M_A
	\end{eqnarray*}
for $k\notin \mathcal{A}_{\rm inc}$, where $\varrho_{\min}$ is the smallest singular value of $A_k$.

Using (G1), we can deduce that $\{\ell(x_k)\}$ is bounded below and $\{h(x_k)\}$ is bounded. Hence,  there exist constants $\ell_{\min}$ and $h_{\max}>0$ such that $\ell_{\min}\leq \ell(x_k)$ and $0\leq h(x_k)\leq h_{\max}$ for all $k$.

From (G3), we can conclude that
	\begin{eqnarray}\label{eq3.2}
		\|P_kH_kP_k\|\leq\|P_k\|^2\|H_k\|\leq M_H.
	\end{eqnarray}

Moreover, from (G1) and (G6), one finds that there exist constants $M_{\lambda},M_{\lambda}^{'}>0$ such that for all $k$
	\begin{eqnarray}\label{eq3.3}
		\|\lambda_k\|\leq M_{\lambda},~~~\|\nabla\lambda_{k}\|\leq M_{\lambda}^{'}.
	\end{eqnarray}
\vspace{-1cm}
\subsection{Preliminary results}
The next lemma provides the reduction in $f(x)$ predicted from the subproblem \eqref{qjt}.
\begin{lemma}\label{lem3.4}
	Suppose that $k\notin\mathcal{A}_{\rm inc}$, and the step $t_k^c$ is the Cauchy step for \eqref{99}. Then
		\begin{eqnarray}\label{lem1}
			f(x_k)-m_{k}^{t}(t_k)\geq \frac{\|P_kg_{k}\|}{6\sqrt{2}}{\rm min}\bigg\{\frac{\|P_kg_{k}\|}{1+\|H_{k}\|}, \frac{1}{2}\sqrt{\frac{\|P_kg_{k}\|}{\sigma_{k}}}\bigg\}
		\end{eqnarray}
	for all $k\geq0$.
\end{lemma}

\begin{proof}
	We can rewrite \eqref{99} as
	\begin{eqnarray}\label{12}
		m_k^{t}(t_k)=f(x_k)+(P_kg_k)^Tt_k+\frac{1}{2}t_k^TH_kt_k+\frac{1}{3}\sigma_k\|t_k\|^3
	\end{eqnarray}
	because of the definition of $P_k$. Hence, the Cauchy step $t_k^c$ for \eqref{99} is
	\begin{eqnarray}\label{tkNN}
		t_k^c=-\beta_k^cP_kg_k~~{\rm and}~~\beta_k^c={\rm arg} {\underset{\beta\in\mathbb{R}^+}{\mbox{min}}}m^t_k(-\beta P_kg_k).
	\end{eqnarray}
	For any $\beta\geq0$, combining the Cauchy-Schwarz and \eqref{tkNN}, one finds that
	\begin{eqnarray}\label{mt}
		&&f(x_k)-m_{k}^{t}(t_k) \nonumber\\
		&\geq& f(x_k)-m_{k}^{t}(-\beta P_kg_k)\nonumber\\
		&=& \beta \|P_kg_{k}\|^{2}-\frac{1}{2}\beta^2(P_kg_{k})^TH_kP_kg_k-\frac{1}{3}\beta^3\sigma_k\|P_kg_k\|^3 \nonumber\\
		&\geq & \beta\|P_kg_{k}\|^{2}(1-\frac{1}{2}\beta \|H_k\|-\frac{1}{3}\beta^2\sigma_k\|P_kg_k\|).
	\end{eqnarray}
	Denote
	\begin{eqnarray*}
		\hat{\beta}_k&=&\frac{3}{2\sigma_k\|P_kg_k\|}\left(-\frac{1}{2}\|H_k\|+\sqrt{\frac{1}{4}\|H_k\|^2
			+\frac{4}{3}\sigma_k\|P_kg_k\|}\right)\\
		&=&2\left(\frac{1}{2}\|H_k\|+\sqrt{\frac{1}{4}\|H_k\|^2
			+\frac{4}{3}\sigma_k\|P_kg_k\|}\right)^{-1}.
	\end{eqnarray*}
	Then for $\beta\in[0, \hat{\beta}_k]$, when $1-\frac{1}{2}\beta \|H_k\|-\frac{1}{3}\beta^2\sigma_k\|P_kg_k\|\geq0$, it follows that $f(x_k)\geq m_{k}^{t}(t_k)$.\\
	Let
	\begin{eqnarray}\label{theta}
		\theta_k:=\frac{1}{\sqrt{2}~\mbox{max}\bigg\{1+\|H_k\|, 2\sqrt{\sigma_k\|P_kg_k\|}\bigg\}}.
	\end{eqnarray}
	By employing the inequalities
	\begin{eqnarray*}
		&&\sqrt{\frac{1}{4}\|H_k\|^2+\frac{4}{3}{\sigma_k\|P_kg_k\|}} \\
		&\leq& \frac{1}{2}\|H_k\|+\frac{2}{\sqrt{3}}\sqrt{\sigma_k\|P_kg_k\|}\\
		&\leq& 2~\mbox{max}\bigg\{\frac{1}{2}\|H_k\|, \frac{2}{\sqrt{3}}\sqrt{\sigma_k\|P_kg_k\|}\bigg\}\\
		&\leq& \sqrt{2}~\mbox{max}\bigg\{1+\|H_k\|, 2\sqrt{\sigma_k\|P_kg_k\|}\bigg\}
	\end{eqnarray*}
	and
	\begin{eqnarray*}
		\frac{1}{2}\|H_k\| \leq \sqrt{2}~\mbox{max}\bigg\{1+\|H_k\|, 2\sqrt{\sigma_k\|P_kg_k\|}\bigg\},
	\end{eqnarray*}
	it follows that $0<\theta_k \leq \hat{\beta}_k$.
	Thus replace $\beta$ in \eqref{mt} with $\theta_k$, we obtain that
	\begin{eqnarray}\label{lema11}
		&&f(x_k)-m_{k}^{t}(t_k) \nonumber\\
		&\geq& \frac{\|P_kg_k\|^2(1-\frac{1}{2}\theta_k\|H_k\|-\frac{1}{3}\theta_k^2\sigma_k\|P_kg_k\|)}
		{\sqrt{2}~\mbox{max}\bigg\{1+\|H_k\|, 2\sqrt{\sigma_k\|P_kg_k\|}\bigg\}}.
	\end{eqnarray}
	Combining the definition \eqref{theta} of $\theta_k$, it follows that $\theta_k\|H_k\| \leq 1$ and $\theta_k^2\sigma_k\|P_kg_k\| \leq 1$. Hence, the numerator of \eqref{lema11} is bounded below by $\frac{1}{6}\|P_kg_k\|^2$, which together with \eqref{lema11}, implies that \eqref{lem1} holds.
\end{proof}

The following result gives a critical bound on the tangential step, which plays a role in following analysis.
	\begin{lemma}\label{lemma4}
	Suppose that (G4) holds. Then the tangential step satisfies
	\begin{eqnarray}\label{tkjie}
		\|t_k\|\leq \sqrt{3}\sqrt{\frac{\|P_kg_k\|}{\sigma_k}},~~~k\geq0.
	\end{eqnarray}
\end{lemma}
\begin{proof}
	Suppose that
	\begin{eqnarray}\label{lem3}
		\|t_k\|> \sqrt{3}\sqrt{\frac{\|P_kg_k\|}{\sigma_k}}
	\end{eqnarray}
	for $k\geq0$. Therefore, from (G4) and \eqref{12}, we have
	\begin{eqnarray*}
		&&m_k^t(t_k)-f(x_k)\\
		&=&(P_kg_k)^Tt_k+\frac{1}{2}t_k^TH_kt_k+\frac{1}{3}\sigma_k\|t_k\|^3\\
		&\geq &-\|t_k\|\|P_kg_k\|+\frac{1}{3}\sigma_k\|t_k\|^3.
	\end{eqnarray*}
	Due to \eqref{lem3}, $\frac{1}{3}\sigma_k\|t_k\|^3-\|t_k\|\|P_kg_k\|>0$. Then $m_k^t(t_k)-f(x_k)>0$, which contradicts \eqref{lem1}. Hence the announced claim follows.
\end{proof}
\vspace{-0.5cm}
\subsection{Feasibility}
\begin{lemma}\label{lem3.1}
	Suppose that Assumptions G hold. Suppose also that $\vert\mathcal{A}\vert<\infty$. Then
	\begin{eqnarray}\label{eq3.4}
		\underset{k\to\infty}{\rm lim}h(x_k)=0.
	\end{eqnarray}
\end{lemma}
\begin{proof}
	Due to $\vert\mathcal{A}\vert<\infty$, one finds that there exists an integer $K_0\in\mathbb{N}$ so that $x_k$ is not acceptable for the filter for all $k>K_0$. Note that $k\notin\mathcal{A}_{\rm inc}\subseteq\mathcal{A}$, which implies that both \eqref{999} and \eqref{lang} hold for $\alpha_k$, for all $k>K_0$. At the same time, from the assumptions (G1)-(G3), and (G6), we get that
	\begin{eqnarray*}
		\|P_kg_k\|\leq M_{pg}
	\end{eqnarray*}
	for all $k$, where $M_{pg}$ is a constant. Hence, along with \eqref{tkjie}, it follows that
	\begin{eqnarray}\label{dkjie}
		&&\|d_k\|\\
		&\leq&\|n_k\|+\|t_k\|\nonumber\\
		&\leq&{\min}\bigg\{\sqrt{3M_{pg}}+\beta_1,\sqrt{3M_{pg}}+\beta_1\beta_2\sigma_k^{-\frac{\beta_3}{2}}\bigg\}\sigma_k^{-\frac{1}{2}}.
	\end{eqnarray}
	
	Then we distinguish two cases, where $k\notin\mathcal{A}$.\\
	$\mathbf{Case~1}$ ($\omega>1$). Using \eqref{999} and Assumptions G, one finds that
	\begin{eqnarray*}
		&&\kappa_h(h(x_k))^{\varsigma}\\
		&<&(-m_k(\alpha_k))^{\omega}\alpha_k^{1-\omega}\sqrt{\sigma_k}^{\omega-1}\\
		&=&[-g_k^Tt_k-\frac{1}{2}\alpha t_k^TH_kt_k-\frac{1}{3}\alpha^2\sigma_k\|t_k\|^3+(\nabla\lambda_k^Td_k)^Tc(x_k)]^{\omega}\alpha_k\sqrt{\sigma_k}^{\omega-1}\\
		&\leq&\left(\|g_k\|\|t_k\|+\frac{1}{2}\|H_k\|\|t_k\|^2+\frac{1}{3}\sigma_k\|t_k\|^3+\|\nabla\lambda_k\|\|d_k\|\|c(x_k)\|\right)^{\omega}
		\alpha_k\sqrt{\sigma_k}^{\omega-1}\\
		&\overset{\eqref{dkjie}}{\leq}&\left(\|g_k\|\sqrt{3M_{pg}}+\frac{\sqrt{3M_{pg}}}{2}\|H_k\|\|t_k\|+M_{pg}\sqrt{3M_{pg}}\right.\\
&&+\left(\sqrt{3M_{pg}}+\beta_1\right)\|\nabla\lambda_k\|\|c(x_k)\| \bigg)^{\omega}\alpha_k\sigma_k^{-\frac{1}{2}}\\		&\overset{\eqref{eq3.3}}{\leq}&\left(\sqrt{3M_{pg}}M_{g}+\frac{3M_{pg}}{2}M_H\sigma_{\min}^{-\frac{1}{2}}+M_{pg}\sqrt{3M_{pg}}+\left(\sqrt{3M_{pg}}+\beta_1\right)M_{\lambda}^{'}M_c\right)^{\omega}
		\alpha_k\sigma_k^{-\frac{1}{2}}\\
		&=&\alpha_k\overline{M}^{\omega}\sigma_k^{-\frac{1}{2}},
	\end{eqnarray*}
	where
$	M_c=\max_{x\in X}\|c(x)\|$
	and
$$	\overline{M}=
		\sqrt{3M_{pg}}M_g+\frac{3M_{pg}}{2}M_H\sigma_{\min}^{-1/2}+M_{pg}\sqrt{3M_{pg}}+\left(\sqrt{3M_{pg}}+\beta_1\right)M_{\lambda}^{'}M_c.
$$
	Combining this result and $1-1/\omega>0$, we can see that
	\begin{eqnarray}\label{eq3.5}
		\left(h(x_k)\right)^{\frac{\varsigma}{\omega}}>\left(\frac{\kappa_h}{\overline{M}^{\omega}}\right)^{1-1/\omega}(\alpha_k\frac{1}{\sqrt{\sigma_k}})^{1/\omega-1}
		\left(h(x_k)\right)^{\varsigma}.
	\end{eqnarray}
	Calling upon \eqref{lang}, it holds that
	\begin{eqnarray*}
		&&\ell(x_k)-\ell(x_{k+1})\\
		&\geq&-\mu m_k(\alpha_k)\\
		&\overset{\eqref{999}}{\geq}&\mu \kappa_h^{\frac{1}{\omega}}(\alpha_k\frac{1}{\sqrt{\sigma_k}})^{1-\frac{1}{\omega}}(h(x_k))^{\frac{\varsigma}{\omega}}\\
		&\overset{\eqref{eq3.5}}{>}&\mu \overline{M}^{1-\omega}\kappa_h(h(x_k))^{\varsigma}.
	\end{eqnarray*}
	$\mathbf{Case~2}$ ($\omega=1$). \eqref{999} and \eqref{lang} together indicate that
	\begin{eqnarray*}
		\ell(x_k)-\ell(x_{k+1})\geq \mu \kappa_h(h(x_k))^{\varsigma}.
	\end{eqnarray*}
	
	No matter what case happens, we can get that
	\begin{eqnarray}\label{eq3.6}
		\ell(x_k)-\ell(x_{k+1})\geq \widetilde{M}(h(x_k))^{\varsigma}\geq0,
	\end{eqnarray}
	where $\widetilde{M}>0$ is a constant. Following (G1), it follows that $\ell(x_k)$ is bounded below, and from \eqref{eq3.6}, $\ell(x_k)$ is also monotonically decreasing for all $k>K_0$. Hence, \eqref{eq3.6} implies that \eqref{eq3.4} holds as $k$ tends to infinity.
\end{proof}

From Lemma  1 of  \cite{Fletcher2002} and its corollary, the following results hold as well.
\begin{lemma}\label{222222}
	Let Assumptions G hold. Suppose also that $\vert\mathcal{A}\vert=\infty$. Then there exists a subsequence $\{k_i\}\subseteq\mathcal{A}$ such that
	\begin{eqnarray}\label{e3.22}
		\underset{i\to \infty}{\rm lim}h(x_{k_i})=0.
	\end{eqnarray}
\end{lemma}
\begin{theorem}\label{them1}
	Let Assumptions G hold. Then
	\begin{eqnarray}
		{\underset{k\rightarrow\infty}{\rm lim}}h(x_k)=0.
	\end{eqnarray}
\end{theorem}
\begin{proof}
	From Lemma \ref{lem3.1} and Lemma \ref{222222}, we can show that the sequence $h(x_k)\to 0$ using the idea of Lemma 8 in \cite{Gu2011}.
\end{proof}
\vspace{-0.5cm}
\subsection{Optimality}
\begin{lemma}\label{lemma3.5}
	Suppose that Assumptions G hold, $k\notin\mathcal{A}_{\rm inc}$, and that \eqref{lem1} holds. Suppose furthermore that
	\begin{eqnarray}\label{38}
		\|P_kg_k\|\geq\epsilon
	\end{eqnarray}
	for a constant $\epsilon>0$ independent of $k$, and that
	\begin{eqnarray}\label{fm}
		\sigma_k\geq\Delta_1:=\frac{(1+M_H)^2}{4\epsilon},
	\end{eqnarray}
	\begin{eqnarray}\label{3.10}
		h(x_k)\leq\frac{\epsilon\sqrt{\epsilon}}{24\sqrt{2}M_{\lambda}^{'}(\sqrt{3M_{pg}}+\beta_1)}.
	\end{eqnarray}
	Then
	\begin{eqnarray}\label{3.11}
		-m_k(\alpha)\geq \alpha\frac{\epsilon\sqrt{\epsilon}}{24\sqrt{2}}\frac{1}{\sqrt{\sigma_k}}
	\end{eqnarray}
	for all $\alpha\in(0,1]$.
\end{lemma}
\begin{proof}
	From Lemma \ref{lem3.4}, \eqref{38} and \eqref{fm}, it follows that
	\begin{eqnarray}\label{3.12}
		-g_k^Tt_k-\frac{1}{2}t_k^TH_kt_k-\frac{1}{3}\sigma_k\|t_k\|^3\geq\frac{\epsilon\sqrt{\epsilon}}{12\sqrt{2}}\frac{1}{\sqrt{\sigma_k}}.
	\end{eqnarray}
	Calling upon (G4), we have that
	\begin{eqnarray}\label{Deltat}
		&&-m_k(\alpha)-\alpha(\nabla\lambda_k^Td_k)^Tc(x_k)\\
		&=&-\alpha g_k^Tt_k-\frac{1}{2}\alpha^2t_k^TH_kt_k-\frac{1}{3}\alpha^3\sigma_k\|t_k\|^3\nonumber\\
		&\overset{0<\alpha\leq1}{\geq}&-\alpha g_k^Tt_k-\frac{1}{2}\alpha t_k^TH_kt_k-\frac{1}{3}\alpha\sigma_k\|t_k\|^3\nonumber\\
		&\overset{\eqref{3.12}}{\geq}&\alpha\frac{\epsilon\sqrt{\epsilon}}{12\sqrt{2}}\frac{1}{\sqrt{\sigma_k}}.
	\end{eqnarray}
	Thus, it follows that
	\begin{eqnarray}\label{Deltan}
		&&-m_k(\alpha)\\
		&\geq&\alpha\left(\frac{\epsilon\sqrt{\epsilon}}{12\sqrt{2}}\frac{1}{\sqrt{\sigma_k}}+(\nabla\lambda_k^Td_k)^Tc(x_k)\right)\nonumber\\
		&\geq&\alpha\left(\frac{\epsilon\sqrt{\epsilon}}{12\sqrt{2}}\frac{1}{\sqrt{\sigma_k}}-\|\nabla\lambda_k\|\|d_k\|c(x_k)\right)\nonumber\\
		&\overset{\eqref{dkjie}}{\geq}&\alpha\frac{1}{\sqrt{\sigma_k}}\left(\frac{\epsilon\sqrt{\epsilon}}{12\sqrt{2}}-(\sqrt{3M_{pg}}+\beta_1)\|\nabla\lambda_k\|c(x_k)\right)\nonumber\\
		&\geq&\alpha\frac{\epsilon\sqrt{\epsilon}}{24\sqrt{2}}\frac{1}{\sqrt{\sigma_k}},
	\end{eqnarray}
	as announced.
\end{proof}

\begin{lemma}\label{111111}
	Suppose that (G1) holds. Then there exist constants $M_{h}$ and $M_{\ell}>1$ for which
	\begin{subequations}\label{3.16}
		\begin{align}
			\ell(x_k+\alpha d_k)-\ell(x_k)-m_k(\alpha)\leq M_\ell\alpha^2\sigma_k^{-1},\tag{\ref{3.16}{a}}\label{3.16a}\\
			h(x_k+\alpha d_k)-(1-\alpha)h(x_k)\leq M_{h}\alpha^2\|d_k\|^2,\tag{\ref{3.16}{b}}\label{3.16b}
		\end{align}
	\end{subequations}
	for all $k\notin\mathcal{A}_{\rm inc}$ and $\alpha\in(0,1]$.
\end{lemma}
\begin{proof}
	It is easy to follow that \eqref{3.16b} holds from Taylor expansions.
	
	From \eqref{eq2.1}, \eqref{qjn} and \eqref{tn}, it follows that
	\begin{eqnarray}
		A_kd_k+c(x_k)=0.
	\end{eqnarray}
	Combining this result and the definition of $\ell(x)$ in \eqref{eq2.2}, we can have that
	\begin{eqnarray}
		&&\ell(x_k+\alpha d_k)-\ell(x_k)\nonumber\\
		&=&f(x_k+\alpha d_k)-\lambda(x_k+\alpha d_k)^Tc(x_k+\alpha d_k)-f_k+\lambda_k^Tc(x_k)\nonumber\\
		&=&\alpha g_k^Td_k+O(\alpha^2\|d_k\|^2)-[\lambda_k+\alpha\nabla\lambda(x_k)^Td_k+O(\alpha^2\|d_k\|^2)]^T\nonumber\\
		&&[c(x_k)+\alpha A_kd_k+O(\alpha^2\|d_k\|^2)]+\lambda_k^Tc(x_k)\nonumber\\
		&=&\alpha g_k^Td_k-(\lambda_k+\alpha\nabla\lambda_k^Td_k)^T(1-\alpha)c(x_k)+\lambda_k^Tc(x_k)+O(\alpha^2\|d_k\|^2)\nonumber\\
		&=&\alpha g_k^T(n_k+t_k)-(1-\alpha)\lambda_k^Tc(x_k)-\alpha(\nabla\lambda_k^Td_k)^Tc(x_k)+\lambda_k^Tc(x_k)+O(\alpha^2\|d_k\|^2)\nonumber\\
		&\overset{\eqref{eq2.3},\eqref{qjn}}{=}&\alpha g_k^Tt_k-\alpha(\nabla\lambda_k^Td_k)^Tc(x_k)+O(\alpha^2\|d_k\|^2).\nonumber
	\end{eqnarray}
	Therefore, it follows \eqref{dkjie} that
	\begin{eqnarray*}
		&&\ell(x_k+\alpha d_k)-\ell(x_k)-m_k(\alpha)\\
		&=&-\frac{1}{2}\alpha^2t_kH_kt_k-\frac{1}{3}\alpha^3\sigma_k\|t_k\|^3+O(\alpha^2\|d_k\|^2)\\
		&\overset{(G4)}{\leq}& \frac{1}{2}\alpha^2t_kH_kt_k+O(\alpha^2\|d_k\|^2)\\
		&=&O(\alpha^2(\frac{1}{\sqrt{\sigma_k}})^2).
	\end{eqnarray*}
	Hence, the \eqref{3.16a} holds.
\end{proof}
\begin{lemma}\label{3.77}
	Let Assumptions G, \eqref{38} and \eqref{3.10} hold, $k\notin\mathcal{A}_{\rm inc}$. Assume also that
	\begin{eqnarray}\label{3.15}
		\sigma_k\geq\Delta_2:={\max} \bigg\{\Delta_1, \left(\frac{24\sqrt{2}M_\ell}{(1-\eta_2)\epsilon\sqrt{\epsilon}}\right)^2\bigg\}.
	\end{eqnarray}
	Then $\rho_k\geq \eta_2$ for all $\alpha\in(0,1]$.
\end{lemma}
\begin{proof}
	Calling upon \eqref{38}, \eqref{3.10} and \eqref{3.15}, it follows that Lemma \ref{lem3.4} and Lemma \ref{lemma3.5} hold. As a result, we can deduce that
	\begin{eqnarray*}
		-m_k(\alpha)\geq \alpha\frac{\epsilon\sqrt{\epsilon}}{24\sqrt{2}}\frac{1}{\sqrt{\sigma_k}}.
	\end{eqnarray*}
	Combining this result, \eqref{3.16a}, the definition of $\rho_k$ and \eqref{3.15}, one finds that
	\begin{eqnarray*}
		&&1-\rho_k\\
		&=&\frac{\ell(x_k+\alpha d_k)-\ell(x_k)-m_k(\alpha)}{-m_k(\alpha)}\\
		&\leq&\frac{24\sqrt{2}M_\ell\alpha^2}{\alpha\epsilon\sqrt{\epsilon}\sqrt{\sigma_k}}\\
		&\leq& 1-\eta_2.
	\end{eqnarray*}
	Therefore, the claim is true.
\end{proof}

\begin{lemma}\label{lemma3.8}
	Let Assumptions G hold and let $\{x_{k_{i}}\}$ be a sequence with $k_i\notin\mathcal{A}_{\rm inc}$. Assume furthermore that \eqref{3.11} follows for a constant $\epsilon>0$ independent of $k_i$ and for all $\alpha\in(0,1]$. Then there exists a constant
	\begin{eqnarray*}
		\bar{\alpha}=\frac{(1-\mu)\epsilon\sqrt{\epsilon}\sqrt{\sigma_{\min}}}{24\sqrt{2}M_\ell}>0
	\end{eqnarray*}
	such that
	\begin{eqnarray}\label{316}
		\ell(x_{k_i}+\alpha d_{k_i})-\ell(x_{k_i})\leq \mu m_{k_i}(\alpha)
	\end{eqnarray}
	for all $k_i$ and $\alpha\leq\bar{\alpha}$.
\end{lemma}
\begin{proof}
	Using \eqref{3.16a} in Lemma \ref{111111}, one finds that
	\begin{eqnarray*}
		&&\ell(x_{k_i}+\alpha d_{k_i})-\ell(x_{k_i})-m_{k_i}(\alpha)\\
		&\leq& M_\ell\alpha^2\sigma_{k_i}^{-1}\\
		&\leq& M_\ell\frac{(1-\mu)\epsilon\sqrt{\epsilon}\sqrt{\sigma_{\min}}}{24\sqrt{2}M_\ell}\alpha\sigma_{k_i}^{-1}\\
		&\leq& \frac{\alpha\epsilon\sqrt{\epsilon}}{24\sqrt{2}}(1-\mu)\sigma_{k_i}^{-\frac{1}{2}}\\
		&\overset{\eqref{3.11}}{\leq}&-(1-\mu)m_{k_i}(\alpha)
	\end{eqnarray*}
	for $\alpha\in(0,\bar{\alpha}]$, which indicates that \eqref{316} follows.
\end{proof}
\vspace{-0.5cm}
\begin{lemma}\label{3.9}
	Suppose that Assumptions G hold and let $\{x_{k_{i}}\}$ be a sequence with $k_i\notin\mathcal{A}_{\rm inc}$. Suppose that \eqref{3.11} holds for all $\alpha\in(0,1]$ and for a constant $\epsilon>0$ independent of $k_i$. Then there exist constants $\nu_1,\nu_2>0$ for which
	\begin{eqnarray*}
		(h(x_{k_i}+\alpha d_{k_i}),\ell(x_{k_i}+\alpha d_{k_i}))\notin \mathcal{F}_{k_i}
	\end{eqnarray*}
	for all $k_i$ and $\alpha\leq {\rm min}\{\nu_1,\nu_2h(x_{k_i})\}$.
\end{lemma}
\begin{proof}
	From the mechanism of Algorithm 2.1, we know that
	\begin{eqnarray}\label{3.17}
		(h(x_{k_i}),\ell(x_{k_i}))\notin \mathcal{F}_{k_i}.
	\end{eqnarray}
	Using \eqref{3.16a} and \eqref{3.11}, it follows that
	\begin{eqnarray*}
		&&\ell(x_{k_i}+\alpha d_{k_i})-\ell(x_{k_i})\\
		&\leq& m_{k_i}(\alpha)+M_\ell\alpha^2\sigma_{k_i}^{-1}\\
		&\leq&-\frac{\alpha\epsilon\sqrt{\epsilon}}{24\sqrt{2}}\sigma_{k_i}^{-\frac{1}{2}}+M_\ell\alpha^2\sigma_{k_i}^{-1}.
	\end{eqnarray*}
	Note that for $\alpha\leq \nu_1:=\frac{\epsilon\sqrt{\epsilon}\sqrt{\sigma_{\min}}}{24\sqrt{2}M_\ell}$, we can get that
	\begin{eqnarray}\label{3.18}
		\ell(x_{k_i}+\alpha d_{k_i})\leq \ell(x_{k_i}).
	\end{eqnarray}
	Similarly, it follows from \eqref{3.16b} that
	\begin{eqnarray}\label{3.19}
		h(x_{k_i}+\alpha d_{k_i})\leq h(x_{k_i})
	\end{eqnarray}
	for $\alpha\leq \nu_2h(x_{k_i})$, where $\nu_2:=\frac{\sigma_{\min}}{M_h(\sqrt{3M_{pg}}+\beta_1)^2}.$
\end{proof}

Combining \eqref{3.17}-\eqref{3.19} and the initialization of the filter, we can deduce that $(h(x_{k_i}+\alpha d_{k_i}),\ell(x_{k_i}+\alpha d_{k_i}))\notin \mathcal{F}_{k_i}$.
\begin{lemma}\label{lemma3.10}
	Suppose that Assumptions G hold and $\vert\mathcal{A}\vert<\infty$. Assume also that \eqref{38} follows for all $k$. Then there exists a constant $\sigma_{\max}>0$ independent of $k$ so that $\sigma_k\leq \sigma_{\max}$ for all k.
\end{lemma}
\begin{proof}
	Due to $\vert\mathcal{A}\vert<\infty$, Lemma \ref{lem3.1} indicates that \eqref{eq3.4} follows. Let $K_1\geq K_0$ be given, which is sufficiently large enough so that $k\notin \mathcal{A}_{\rm inc}$ follows for all $k\geq K_1$.
	
	To obtain a contradiction, we assume that the iteration $j$ is the first iteration after $K_1$ such that
	\begin{eqnarray}\label{eq3.20}
		\sigma_j\geq \gamma_2\Delta_3
	\end{eqnarray}
	with $\Delta_3:= {\max} \left\{\Delta_2, \sigma_{K_1}\right\}$.  \eqref{eq3.20} implies that $\sigma_j\geq\gamma_2\sigma_{K_1}$. This result guarantees that $j\geq K_1+1$. Hence, one finds that $j-1\geq K_1$ and we can deduce that $j-1\notin \mathcal{A}_{\rm inc}$. Calling upon step 11 of Algorithm 2.1, it follows from \eqref{eq3.20} that
	\begin{eqnarray*}
		\sigma_{j-1}\geq \Delta_3\geq \Delta_2.
	\end{eqnarray*}
	Thus, the result follows from Lemma \ref{3.77}.
	
	Moreover, it follows from $\Delta_2\geq\Delta_1$, \eqref{3.10} and \eqref{38} that Lemma \ref{lemma3.5} is applicable. Therefore, we can deduce that Lemma \ref{3.9} is applicable, which indicates that $x_{j-1}+\alpha d_{j-1}$ is accepted by the filter, for all $\alpha\leq \min \left\{\nu_1, \nu_2h(x_{j-1}) \right\}$. It follows from this result, $\rho_{j-1}\geq\eta_2$ and the mechanism of the Algorithm 2.1 that
	\begin{eqnarray*}
		\sigma_{j-1}\geq\sigma_j\geq \gamma_2 {\max} \left\{\Delta_2, \sigma_{K_1}\right\},
	\end{eqnarray*}
	which contradicts the fact that $j$ is the first iteration after $K_1$ such that \eqref{eq3.20} follows. Hence, for all $k\geq K_1$, one finds that $\sigma_k\leq\gamma_2\Delta_3$. If we define
	\begin{eqnarray*}
		\sigma_{\max}=\max \left\{ \sigma_0,\dots,\sigma_{K_1},\gamma_2\Delta_3 \right\},
	\end{eqnarray*}
	the desired conclusion follows.
\end{proof}
\vspace{-0.5cm}
\begin{lemma}\label{lemma3.11}
Suppose that Assumptions G hold. Then
	\begin{eqnarray}\label{3.21}
		h(x_k)=0\Rightarrow -m_k(\alpha)>0,
	\end{eqnarray}
	and for all $k$ and $\alpha\in(0,1]$,
	\begin{eqnarray}\label{3.22}
		\mathcal{H}_k:={\rm min}\{h:(h,\ell)\in \mathcal{F}_k\}>0.
	\end{eqnarray}
\end{lemma}
\begin{proof}
	From \eqref{qjn}, one finds that
	\begin{eqnarray}\label{3.23}
		&&\|n_k\|\nonumber\\
		&=&\|A_k^T(A_kA_k^T)^{-1}c(x_k)\|\nonumber\\
		&\leq&\|A_k^T(A_kA_k^T)^{-1}\|h(x_k)\nonumber\\
		&\overset{(G6)}{\leq}&\frac{1}{M_A}h(x_k).
	\end{eqnarray}
	Hence, using \eqref{3.23}, $h(x_k)=0$ yields that $n_k=0$ and $c(x_k)=0$, which indicates that $d_k=t_k$.
	
	Moreover, we know that $\|P_kg_k\|>0$ or the algorithm would stop in step 2. It follows from (G4) that
	\begin{eqnarray*}
		&&-m_k(\alpha)\\
		&=&-\alpha g_k^Tt_k-\frac{1}{2}\alpha^2t_k^TH_kt_k-\frac{1}{3}\alpha^3\sigma_k\|t_k\|^3+\alpha(\nabla\lambda_k^Td_k)^Tc(x_k)\\
		&=&-\alpha g_k^Tt_k-\frac{1}{2}\alpha^2t_k^TH_kt_k-\frac{1}{3}\alpha^3\sigma_k\|t_k\|^3\\
		&\overset{0<\alpha<1}{\geq}&\alpha\bigg(-g_k^Tt_k-\frac{1}{2}t_k^TH_kt_k-\frac{1}{3}\sigma_k\|t_k\|^3\bigg)\\
		&\overset{\eqref{lem1}}{\geq}&\frac{\alpha\|P_kg_{k}\|}{6\sqrt{2}}{\rm min}\bigg\{\frac{\|P_kg_{k}\|}{1+\|H_{k}\|}, \frac{1}{2}\sqrt{\frac{\|P_kg_{k}\|}{\sigma_{k}}}\bigg\}>0.
	\end{eqnarray*}
	Hence, \eqref{3.21} follows.
	
	Next, we establish the second conclusion. Since $h_{\rm max}>0$, for $k=0$, calling upon the mechanism of Algorithm 2.1, one can show that the claim follows.
	
	Let the claim holds for $k$. Provided that $h(x_k)>0$, and $x_k$ is accepted by the filter, we can get that $\mathcal{H}_{k+1}>0$ in view of $\gamma_h\in(0,1)$.
	
	If $h(x_k)=0$, it follows from \eqref{3.21} that $-m_k(\alpha)>0$ for all $\alpha\in(0,1]$. Hence, for all trial step sizes, we have that \eqref{999} holds. Thus, our algorithm always consider step 8. Furthermore, $\alpha_k$ satisfies \eqref{lang}. Hence, $x_k$ is not accepted by the filter, which shows that $\mathcal{H}_{k+1}=\mathcal{H}_{k}>0$, as announced.
\end{proof}
\vspace{-0.5cm}
\begin{lemma}\label{lemma3.12}
	Suppose that Assumptions G hold. Suppose also that $\vert\mathcal{A}\vert<\infty$. Then
	\begin{eqnarray*}
		{\underset {k\rightarrow\infty}{\rm lim}}~ \|P_kg_k\|=0.
	\end{eqnarray*}
\end{lemma}
\begin{proof}
	To derive a contradiction, we assume that there exists a subsequence $\{x_{k_i}\}$ so that \eqref{38} follows, namely, $\|P_{k_i}g_{k_i}\|>\epsilon$ for all $i$.
	
	Due to $\vert\mathcal{A}\vert<\infty$, for all $k_i\geq K_2$, there exists an integer $K_2\geq K_1\geq K_0$ so that $k_i\notin \mathcal{A}$. From the mechanism of Algorithm 2.1, one finds that \eqref{lang} and \eqref{999} follow for all  $k_i\geq K_2$. Consequently, the above results indicate that
		\begin{eqnarray}\label{3.25}
			{\underset {i\rightarrow\infty}{\rm lim}}~ m_{k_i}(\alpha_{k_i})=0
		\end{eqnarray}
	because $\ell(x_{k_i})$ is monotonically decreasing and bounded below from \eqref{eq3.6}.
	
	Combining \eqref{lem1}, Lemma \ref{lemma3.10}, and $\|P_kg_k\|>\epsilon$, it follows that
		\begin{eqnarray}\label{3.26}
			&&-g_{k_i}^Tt_{k_i}-\frac{1}{2}t_{k_i}^TH_{k_i}t_{k_i}-\frac{1}{3}\sigma_k\|t_{k_i}\|^3\nonumber\\
			&\geq &\frac{\|P_{k_i}g_{k_i}\|}{6\sqrt{2}}{\rm min}\left\{\frac{\|P_{k_i}g_{k_i}\|}{1+\|H_{k_i}\|}, \frac{1}{2}\sqrt{\frac{\|P_{k_i}g_{k_i}\|}{\sigma_{k_i}}}\right\}\nonumber\\
			&\geq&\tilde{\Delta},
		\end{eqnarray}
	where $$\tilde{\Delta}:=\frac{\epsilon}{6\sqrt{2}}{\rm min}\left\{\frac{\epsilon}{1+M_H}, \frac{\sqrt{\epsilon}}{2}\frac{1}{\sqrt{\sigma_{\max}}}\right\}.$$

	Then, we can get that
		\begin{eqnarray*}
			&&-m_{k_i}(\alpha_{k_i})-\alpha_{k_i}(\nabla\lambda_{k_i}^Td_{k_i})^Tc_{k_i}\\
			&=&-\alpha_{k_i}g_{k_i}^Tt_{k_i}-\frac{1}{2}\alpha_{k_i}^2t_{k_i}^TH_{k_i}t_{k_i}
			-\frac{1}{3}\alpha_{k_i}^3\sigma_{k_i}\|t_{k_i}\|^3\\
			&\geq&\alpha_{k_i}\bigg(-g_{k_i}^Tt_{k_i}-\frac{1}{2}t_{k_i}^TH_{k_i}t_{k_i}-\frac{1}{3}\sigma_{k_i}\|t_{k_i}\|^3\bigg)\\
			&\overset{\eqref{3.26}}{\geq}&\tilde{\Delta}\alpha_{k_i}
		\end{eqnarray*}
	for $k_i\geq K_2$. As a consequence, it follows that
	\begin{eqnarray}\label{3.28}
		-m_{k_i}(\alpha_{k_i})\geq(\tilde{\Delta}+(\nabla\lambda_{k_i}^Td_{k_i})^Tc_{k_i})\alpha_{k_i}.
	\end{eqnarray}
	Meanwhile, Lemma \ref{lem3.1} provides that there exists an integer $K_3\geq K_2$ so that $h(x_{k_i})\leq \frac{\tilde{\Delta}\sqrt{\sigma_{\min}}}{2(\sqrt{3M_{pg}}+\beta_1)M_{\lambda}^{'}}$ holds for all $k_i\geq K_3$. This result and \eqref{3.28} yield that
	\begin{eqnarray}\label{3.29}
		&&-m_{k_i}(\alpha_{k_i})\nonumber\\
		&\geq&[\tilde{\Delta}-(\nabla\lambda_{k_i}^Td_{k_i})^Tc_{k_i}]\alpha_{k_i}\nonumber\\
		&\geq&[\tilde{\Delta}-\|\nabla\lambda_{k_i}\|\|d_{k_i}\|h(x_{k_i})]\alpha_{k_i}\nonumber\\
		&\geq&\frac{1}{2}\tilde{\Delta}\alpha_{k_i}
	\end{eqnarray}
	for $k_i\geq K_3$. The last inequality and \eqref{3.25} show that ${\rm lim}_{i\rightarrow\infty}\alpha_{k_i}=0$.
	
	In general, suppose that $K_3$ is large enough so that $\alpha_{k_i}<1$. Hence, $\alpha_{k_i,0}=1$ cannot be accepted. Moreover, it follows from $k_i\notin\mathcal{A}$ and $\alpha_{k_{i},l_i}>\alpha_{k_i}$ that last rejected trial step size
	\begin{eqnarray}\label{3.30}
		\alpha_{k_i,l_i}\in[\alpha_{k_i}/\omega_2,\alpha_{k_i}/\omega_1]
	\end{eqnarray}
	satisfies \eqref{999}. As a consequence, $\alpha_{k_i,l_i}$ can not be accepted since \eqref{lang} does not hold, namely,
	\begin{eqnarray}\label{3.31}
		\ell(x_{k_i}+\alpha_{k_i,l_i}d_{k_i})-\ell(x_{k_i})>\mu[m_{k_i}(\alpha_{k_i,l_i})],
	\end{eqnarray}
	or it is not accepted by the current filter, namely,
	\begin{eqnarray}\label{3.32}
		(h(x_{k_i}+\alpha_{k_i,l_i}d_{k_i}),\ell(x_{k_i}+\alpha_{k_i,l_i}d_{k_i}))\in \mathcal{F}_{k_i}=\mathcal{F}_{K_2}.
	\end{eqnarray}
	Now, one provides that either \eqref{3.31} does not hold or $(h(x_{k_i}+\alpha_{k_i,l_i}d_{k_i}),\ell(x_{k_i}+\alpha_{k_i,l_i}d_{k_i}))\notin \mathcal{F}_{k_i}$ for sufficiently large $k_i$.
	
	Consider \eqref{3.31}. Combining \eqref{3.30} and ${\lim}_{i\to \infty}\alpha_{k,i}=0$, we can get that ${\lim}_{i\rightarrow\infty}\alpha_{k_i,l_i}=0$. Consequently, it follows from Lemma \ref{lemma3.8} that $\alpha_{k_i,l_i}\leq \bar{\alpha}$ for sufficiently large $k_i$, which indicates that \eqref{3.31} does not hold for those $k_i$.
	
	Consider \eqref{3.32}. Denote $\mathcal{H}_{K_2}={ \min}\{h:(h,\ell)\in \mathcal{F}_{K_2}\}$. Lemma \ref{lemma3.11} shows that $\mathcal{H}_{K_2}>0$. Our assumptions together \eqref{3.16b} imply that
	\begin{eqnarray*}
		h(x_{k_i}+\alpha_{k_i,l_i}d_{k_i})\leq(1-\alpha_{k_i,l_i})h(x_{k_i})+M_{h}\alpha_{k_i,l_i}^2(\sqrt{3M_{pg}}+\beta_1)^2\sigma_{\min}^{-1}.
	\end{eqnarray*}
	It follows from the above inequality, ${\rm lim}_{i\rightarrow\infty}\alpha_{k_i,l_i}=0$ and ${\rm lim}_{i\rightarrow\infty}h(x_{k_i})=0$ that $h(x_{k_i}+\alpha_{k_i,l_i}d_{k_i})<\mathcal{H}_{K_2}$ for $k_i$ sufficiently large. which is contradiction with \eqref{3.32}. Thus, the desired conclusion follows.
\end{proof}

Then, we can get the following result from Lemma 3.1 in \cite{Fletcher2002GLOBAL}.
\begin{lemma}\label{lemma3.13}
	Suppose that (G1) and \eqref{3.1} hold. Then
	\begin{eqnarray}\label{3.33}
		\|n_k\|\geq\frac{1}{\kappa_{hn}}h(x_k)
	\end{eqnarray}
	for a constant $\kappa_{hn}$ independent of $k$.
\end{lemma}
\begin{lemma}\label{lemma3.14}
Suppose that Assumptions G hold and let $\{x_{k_i}\}$ be a sequence with $\|P_{k_i}g_{k_i}\|\geq \epsilon$ for a constant $\epsilon>0$ independent of $k_i$. Then there exists $K\in\mathbb{N}$ so that for all $k_i\geq K$, $k_i\notin \mathcal{A}$.
\end{lemma}
\vspace{-0.15cm}
\begin{proof}
	Using Theorem \ref{them1}, it follows that ${\rm lim}_{i\rightarrow\infty}h(x_{k_i})=0$, which indicates that $k_i\notin \mathcal{A}_{\rm inc}$ for $k_i\geq K_3$.
	Then we consider two cases.\\
	$\mathbf{Case~1}$ (there is a constant $\tilde{\kappa}>0$ independent of $k_i$ for which $\sigma_{k_i}\leq \tilde{\kappa}$ for all $k_i$).\\
	By the same argument of \eqref{3.29}, one finds that there is a constant $\bar{\Delta}>0$ independent of $k_i$ for which
	\begin{eqnarray}\label{3.34}
		-m_{k_i}(\alpha)\geq \bar{\Delta}\alpha
	\end{eqnarray}
	for $k_i$ large enough and $\alpha\in(0,1]$.
	
	In general, suppose that \eqref{3.34} follows for all $k_i$. Combining Lemma \ref{lemma3.8} and Lemma \ref{3.9}, it follows that the constants $\bar{\alpha},\nu_1,\nu_2>0$ are existent.
	
	Choose $K\in \mathbb{N}$ with $K\geq K_3$ such that, for all $k_i\geq K$,
	\begin{eqnarray}\label{3.35}
		h(x_{k_i})<{\rm  min}\bigg\{\delta_{h},\frac{\bar{\alpha}}{\nu_2},\frac{\nu_1}{\nu_2},\bigg(\frac{\bar{\Delta}^{\omega}\omega_1\nu_2}{\kappa_h}\bigg)^{\frac{1}{\varsigma-2}}\bigg\},
	\end{eqnarray}
	where $\omega_1$ is from step 10 in Algorithm 2.1. The following part of this proof is in the same way as Lemma 10 in \cite{2005Line}, which proves the claim.\\
	$\mathbf{Case~2}$ (there exist subsequences of $\{k_i\}$ such that $\sigma_{k_i}$ tends to infinity).
	
	Similarly, we can choose $K\in \mathbb{N}$ with $K\geq K_3$ for which
	\begin{eqnarray}\label{3.36}
		h(x_{k_j})<{\rm  min}\bigg\{\delta_{h},\frac{\bar{\alpha}}{\nu_2},\frac{\nu_1}{\nu_2},\frac{\epsilon\sqrt{\epsilon}}{24\sqrt{2}M_{\lambda}^{'}(\sqrt{3M_{pg}}+\beta_1)},\bigg[\frac{(\epsilon\sqrt{\epsilon})^{\omega}\omega_1\nu_2}{(24\sqrt{2})^{\omega}\kappa_{hn}\kappa_h}\bigg]^{\frac{1}{\varsigma-2}}\bigg\}
	\end{eqnarray}
	for all $k_j\geq K$. For the sake of brevity, let $\{k_j\}$ be a subsequence of $\{k_i\}$ for which
	\begin{eqnarray*}
		{\underset {j\rightarrow\infty}{\rm lim}}~ \frac{1}{\sqrt{\sigma_{k_j}}}=0.
	\end{eqnarray*}
	Hence, for $k_j$ sufficiently large, we can get that $\sigma_{k_j}\geq \Delta_{1}$ for the constant $\Delta_1>0$ in \eqref{fm}. So \eqref{3.11} follows for those $k_j$, namely,
	\begin{eqnarray}\label{3.37}
		-m_{k_j}(\alpha)\geq \alpha\frac{\epsilon\sqrt{\epsilon}}{24\sqrt{2}}\frac{1}{\sqrt{\sigma_{k_j}}}.
	\end{eqnarray}
	Without loss of generality, suppose that \eqref{3.37} follows for those $k_j$.
	
	From Lemma \ref{lemma3.11}, for all $k_j\geq K$ with $h(x_{k_j})=0$, one finds that both \eqref{999} and \eqref{lang} hold in iteration $k_j$. Hence, we can see that $k_j\notin \mathcal{A}$.
	
	For those iterations $k_j\geq K$ with $h(x_{k_j})>0$, it follows from \eqref{3.36} that $k_j\notin \mathcal{A}_{\rm inc}$,
	\begin{eqnarray}\label{3.38}
		\frac{(24\sqrt{2})^{\omega}\kappa_{hn}\kappa_h[h(x_{k_j})]^{\varsigma-1}}{(\epsilon\sqrt{\epsilon})^{\omega}}<\omega_1\nu_2h(x_{k_j})
	\end{eqnarray}
	and
	\begin{eqnarray}\label{3.39}
		\nu_2h(x_{k_j})<{\rm min}\{\bar{\alpha},\nu_1\}.
	\end{eqnarray}
	For an arbitrary $k_j$ such that $h(x_{k_j})>0$, we can define that
	\begin{eqnarray}\label{3.40}
		\zeta_{k_j}=\nu_2h(x_{k_j})\overset{\eqref{3.39}}{=}{\rm min}\{\bar{\alpha},\nu_1,\nu_2h(x_{k_j})\}.
	\end{eqnarray}
	Using Lemma \ref{lemma3.8} and Lemma \ref{3.9}, for $\alpha_{k_j,l}\leq \zeta_{k_j}$, we can conclude that both
	\begin{eqnarray}\label{3.41}
		(h(x_{k_j}+\alpha_{k_j,l_j}d_{k_j}),\ell(x_{k_j}+\alpha_{k_j,l_j}d_{k_j}))\notin \mathcal{F}_{k_j}
	\end{eqnarray}
	and
	\begin{eqnarray}\label{3.42}
		\ell(x_{k_j}+\alpha_{k_j,l_j}d_{k_j})-\ell(x_{k_j})\leq \mu [m_{k_j}(\alpha_{k_j,l_j})]
	\end{eqnarray}
	hold. Let $\alpha_{k_j,L}$ be the first trial step size such that \eqref{3.41} and \eqref{3.42} hold. The step 10 in Algorithm 2.1 provides that
	\begin{eqnarray*}
		\alpha\geq \omega_1\zeta_{k_j}\overset{\eqref{3.40}}{=}\omega_1\nu_2h(x_{k_j})\overset{\eqref{3.38}}{>}\frac{(24\sqrt{2})^{\omega}\kappa_{hn}\kappa_h[h(x_{k_j})]^{\varsigma-1}}
		{(\epsilon\sqrt{\epsilon})^{\omega}}
	\end{eqnarray*}
	for $\alpha\geq \alpha_{k_j,L}$. Moreover, \eqref{3.33} yields that
	\begin{equation*}
		\frac{1}{\sqrt{\sigma_{k_j}}}\geq\|n_{k_j}\| \geq  \frac{1}{\kappa_{hn}}h(x_{k_j}).
	\end{equation*}
	From above results, one finds that
	\begin{eqnarray}\label{3.43}
		&&[-m_{k_j}(\alpha)]^{\omega}\bigg(\alpha\frac{1}{\sqrt{\sigma_{k_j}}}\bigg)^{1-\omega}\nonumber\\
		&\overset{\eqref{3.37}}{\geq}&
		\bigg(\frac{\epsilon\sqrt{\epsilon}}{24\sqrt{2}}\bigg)^{\omega}\alpha(\sqrt{\sigma_{k_j}})^{-1}\nonumber\\
		&>&\kappa_h[h(x_{k_j})]^{\varsigma}
	\end{eqnarray}
	for $\alpha\geq \alpha_{k_j,L}$.
	
	From \eqref{3.43}, we know that the algorithm goes to step 8. Furthermore, \eqref{3.42} implies that \eqref{lang} follows for $\alpha_{k_j,L}$. It remains to verify
	\begin{eqnarray*}
		\alpha_{k_j,L}\geq \alpha_{k_j}^{\rm min}.
	\end{eqnarray*}
	Following \eqref{3.43}, one finds that
	\begin{eqnarray*}
		&&\alpha_{k_j,L}\\
		&\geq&\frac{\kappa_h[h(x_{k_j})]^{\varsigma}(\sqrt{\sigma_{k_j}})^{1-\omega}}{[-g_{k_j}^Tt_{k_j}-
			\frac{1}{2}\alpha_{k_j,L}t_{k_j}^TH_{k_j}t_{k_j}-\frac{1}{3}\sigma_{k_j}\alpha_{k_j,L}^2\|t_{k_j}\|^3+(\nabla\lambda_{k_j}^Td_{k_j})^Tc_{k_j}]^{\omega}}\\
		&\geq& \frac{\kappa_h[h(x_{k_j})]^{\varsigma}(\sqrt{\sigma_{k_j}})^{1-\omega}}{[-g_{k_j}^Tt_{k_j}+
			(\nabla\lambda_{k_j}^Td_{k_j})^Tc_{k_j}]^{\omega}},
	\end{eqnarray*}
	which together \eqref{alphamin} imply that $\alpha_{k_j,L}\geq \alpha_{k_j}^{\rm min}$.
	
	Therefore, the algorithm does not go to the feasibility restoration procedure. Following this result and \eqref{3.41}-\eqref{3.43}, one finds that $\alpha_{k_j,L}$ is the accepted step size $\alpha_{k_j}$. Hence, for all $k_j\geq K$, $k_j\notin \mathcal{A}$ holds because of the mechanism of Algorithm 2.1.
	
	Thus, no matter Case 1 or Case 2 happens, we know that the desired conclusion follows.
\end{proof}

The following theorem gives global convergence conclusion.

\begin{theorem}\label{theorem3.15}
	Suppose that Assumptions G hold. Then
	\begin{eqnarray*}
		{\underset {k\rightarrow\infty}{\rm lim\,inf}}~\|P_kg_k\|=0.
	\end{eqnarray*}
\end{theorem}
\begin{proof}
	To prove our claim, we consider two cases.\\
	$\mathbf{Case~1}$ ($\vert\mathcal{A}\vert<\infty$). Lemma \ref{lemma3.12} has proven this claim.\\
	$\mathbf{Case~2}$ (There exists a subsequence $\{x_{k_i}\}$ for which $k_i\in \mathcal{A}$).
	
	To obtain a contradiction, assume that ${\rm lim\,sup}_{k\rightarrow\infty}\|P_kg_k\|>0$. Hence, there exist a subsequence $\{x_{k_{i_j}}\}$ so that $\|P_{k_{i_j}}g_{k_{i_j}}\|>\epsilon$ for a constant $\epsilon>0$ independent of $k_{i_j}$ and all $k_{i_j}$. Then, from Lemma \ref{lemma3.14}, we know that there is an iteration $k_{i_j}$ such that $k_{i_j}\notin \mathcal{A}$, which is in contradiction with the choice of $\{x_{k_{i_j}}\}$ such that
	${\rm lim }_{i\rightarrow\infty}\|P_{k_i}g_{k_i}\|=0$. Therefore, the required result holds.
\end{proof}
\vspace{-0.5cm}
\section{Numerical Results}\label{sec4}
In this section, we present numerical results to show the efficiency of LsFSARC (Algorithm 2.1). We performed it with Intel(R) Core(TM) i5-6200U CPU @ 2.30GHz 2.40GHz. Numerical testing was implemented in MATLAB version 9.4.0.813654 (R2018a).

In our implementation, the parameters: $\varepsilon=10^{-6}$, $\beta_1=0.1$, $\beta_2=100$, $\beta_3=0.01$, $\gamma_h=10^{-5}$, $\kappa_h=10^{-4}$, $\tau_2=2.01$, $\tau_1=2$, $\eta_1=0.01,~\eta_2=0.9$. The algorithm terminates when $${\rm Res}:=\max \{\|P_kg_k\|,\|c(x_k)\|\}\leq\varepsilon$$ is satisfied.
The numerical results are presented in Table \ref{table5.1}. The test problems are from CUTEst collection \cite{CUTEst}. $n$ denotes the number of variables. $m$ denotes the number of equality constraints.   NF and NC are the numbers of computation of the objective function and constraint function, respectively. NIT denotes the numbers of iterations. The numbers of computation of the objective function's gradient is denoted by NG. The CPU times in Table \ref{table5.1} are counted in seconds.

For comparison, we include the corresponding results obtained from Algorithm 2.1 (Alg. 2.1)  in \cite{Chen} and Algorithm 2.2 (Alg. 2.2) in \cite{liu2011SIOPT}. The comparison numerical results are reported in Table \ref{table5.5} and Table \ref{table5.6}, respectively. The numerical results of LANCELOT are also from literatures  \cite{Chen} and  \cite{liu2011SIOPT}, respectively. Furthermore, to display the performance based on the numerical results in Table \ref{table5.5} and Table \ref{table5.6}  visually, we use the logarithmic performance profiles \cite{Dolan} (see Fig \ref{Figure1} and Fig \ref{Figure2}). From Table \ref{table5.5}, Table \ref{table5.6},  Fig \ref{Figure1}, and Fig \ref{Figure2}, it can be seen that LsFSARC can be comparable with those algorithms for the given problems.
~\\
\begin{center}
	{
		\begin{longtable}{ccccccccccccccccllll}
			\caption{Numerical results of the LsFSARC}\label{table5.1}
			\endfirsthead
			\multicolumn{10}{l}{Table \ref{table5.1} continued}\\
			\hline
			\multirow{2}{*}{Problem}  & \multicolumn{2}{l}{Dimension}& \multirow{2}{*}{ NIT}&\multirow{2}{*}{NF}&\multirow{2}{*}{NC}&\multirow{2}{*}{NG}&\multirow{2}{*}{Res }&\multirow{2}{*}{CPU-Time }  \\
			\cline{2-3}
			&$n$&$m$                       & & &                                   &        &&&\\
			\hline
			\endhead
			\hline
			\multirow{2}{*}{Problem}  & \multicolumn{2}{l}{Dimension}& \multirow{2}{*}{ NIT}& \multirow{2}{*}{ NF}&\multirow{2}{*}{NC} &\multirow{2}{*}{NG}&\multirow{2}{*}{Res }&\multirow{2}{*}{CPU-Time }  \\
			\cline{2-3}
			&$n$&$m$                       & & &                                &        &&& \\
			\hline
			AIRCRFTA &8     &5        &2    &3  &3  &3  &1.5932e$-$08   &0.0164\\
			ARGTRIG &200   &200       &3    &3 &4   &4  &6.8423e$-$07   &3.8064 \\
			BDVALUE  &102   &100      &2     &3 &3  &3    &9.5721e$-$10   &0.1291  \\
			BOOTH    &2     &2        &1     &2  &2  &2    &0.0000e$+$00   &0.0074   \\
			BROYDN3D &500   &500     &4 &5     &5 &5     &1.0634e$-$09   &18.2427 \\
			BT1     &2     &1        &5 &5     &6   &6   &1.3889e$-$07   &0.0234 \\
			BT2     &3     &1        &9  &9     &10  &10   &7.2220e$-$10   &0.0297\\
			BT3     &5     &3        &3 &4     &4 &4     &1.0934e$-$10   &0.0084 \\
			BT4     &3     &2        &5 &5     &1   &6   &2.5683e$-$07   &0.0200  \\
			BT5     &3     &2        &7  &8     &8  &8    &2.2452e$-$08   &0.0147  \\
			BT6     &5     &2        &12 &15    &13  &13   &4.2933e$-$07   &0.0277  \\
			BT7     &5     &3        &5  &4     &6  &6    &2.9464e$-$07   &0.0334  \\
			BT8     &5     &2        &6  &6     &7  &7    &3.5654e$-$07   &0.0213 \\
			BT9     &4     &2        &9  &11    &10 &10    &8.0670e$-$08   &0.0198  \\
			BT10    &2     &2        &2 &2     &3  &3    &2.0895e$-$09   &0.0187 \\
			BT11    &5     &3        &9 &9     &10  &10   &5.2740e$-$09   &0.0349  \\
			BT12    &5     &3        &8  &8     &9  &9    &3.5613e$-$07   &0.0315 \\
			BYRDSPHR &3     &2       &8  &8     &9  &9    &3.5170e$-$13   &0.0346  \\
			CLUSTER  &2     &2       &4 &7     &5   &5   &3.7799e$-$10   &0.0174   \\
			DTOC3    &299   &198     &27 &28    &28 &28    &3.9893e$-$07   &42.3079                \\
			DTOC4    &299   &198     &3 &4     &4   &4   &2.3783e$-$07   &17.4360                \\
			DTOC5    &19    &9       &0  &1     &1  &1    &0.0000e$+$00   &0.0009                \\
			GENHS28  &10    &8       &5 &5     &6  &6    &3.3345e$-$08   &0.0360 \\
			GOTTFR  &2     &3        &5 &8     &6   &6   &2.1823e$-$10   &0.0169  \\
			HAGER1   &1001  &500     &16 &16    &17 &17    &5.8785e$-$07   &102.1731                 \\
			HAGER2   &1001  &500     &12 &12    &13 &13    &4.6752e$-$07   &71.8340                 \\
			HAGER3   &1001  &500     &10 &10    &11 &11    &2.9516e$-$07   &52.7732                 \\
			HATFLDF  &3     &3       &3  &3     &4  &4    &1.7795e$-$08   &0.0387 \\
			HATFLDG  &25    &25      &2 &2     &3   &3   &4.0738e$-$12   &0.0684                 \\
			HEART8   &8     &8       &5  &6     &6  &6    &1.8500e$-$07   &0.2510                 \\
			HIMMELBA  &2     &2      &1  &1     &2 &2     &5.3134e$-$07   &0.0108  \\
			HIMMELBC  &2     &2      &2  &2     &3  &3    &2.8424e$-$13   &0.0260   \\
			HIMMELBE  &3     &3      &5  &1     &2  &6    &7.4308e$-$07   &0.0121  \\
			HS06    &2     &1        &13  &13    &14  &14   &6.0080e$-$08   &0.0232 \\
			HS07    &2     &1        &7  &8     &8   &8   &5.4175e$-$08   &0.0273 \\
			HS08    &2     &2        &2 &2     &3   &3   &2.5421e$-$13   &0.0105    \\
			HS09    &2     &1        &6  &7     &7   &7   &3.9241e$-$07   &0.0112   \\
			HS26    &3     &1        &9  &10    &10   &10  &1.2431e$-$08   &0.0189 \\
			HS27    &3     &1        &26  &29    &27   &27  &3.4457e$-$08   &0.0447  \\
			HS28    &3     &1        &5  &7     &6   &6   &1.9369e$-$08   &0.0081  \\
			HS39    &4     &2        &9 &11    &10   &10  &8.1036e$-$08   &0.0350  \\
			HS40    &4     &3        &19   &19    &20 &120    &9.5368e$-$07   &0.0527  \\
			HS42    &4     &2        &28   &52    &29 &29    &8.4021e$-$07   &0.0657  \\
			HS46    &5     &2        &12  &13    &13  &13   &8.8987e$-$07   &0.0551  \\
			HS47    &5     &3        &20  &20    &21  &21   &5.8020e$-$07   &0.1012 \\
			HS48    &5     &2        &4  &5     &5   &5   &8.6615e$-$08   &0.0169 \\
			HS49    &5     &2        &22  &23    &23  &23   &7.8456e$-$07   &0.0975\\
			HS50    &5     &3        &12  &15    &13  &13   &1.3377e$-$07   &0.0667 \\
			HS51    &5     &3        &3  &5     &4   &4   &9.8047e$-$15   &0.0143  \\
			HS52    &5     &3        &6   &6     &7   &7   &3.1612e$-$10   &0.0209 \\
			HS56    &7     &4        &0 &1     &1  &1    &0.0000e$+$00   &0.0043    \\
			HS61    &3     &2        &6   &6     &7  &7    &6.3198e$-$07   &0.0232 \\
			HS77    &5     &2        &11   &13    &12 &12    &7.8820e$-$07   &0.0267 \\
			HS78    &5     &3        &12  &14    &13   &13  &5.3810e$-$07   &0.0303 \\
			HS79    &5     &3        &8  &8     &9   &9   &8.8824e$-$07   &0.0353 \\
			HS100LNP  &7     &2      &15   &21    &16   &16  &4.9734e$-$07   &0.0540 \\
			HS111LNP  &10    &3      &12  &12    &13   &13  &8.0141e$-$08   &0.0587 \\
			HYPCIR    &2     &2      &1   &1     &2  &2    &5.4209e$-$07   &0.0118  \\
			INTEGREQ  &5     &5      &1  &1     &2    &2  &3.8263e$-$07   &0.0117           \\
			MARATOS   &2     &1      &3  &4     &4  &4    &2.6776e$-$07   &0.0072 \\
			MWRIGHT   &5     &3      &8  &10    &9   &9   &6.1967e$-$09   &0.0267 \\
			ORTHREGB  &27    &6     &7      &8       &8       &8 &1.1322e$-$08        &0.2326           \\
			POWELLSQ  &2     &2      &1   &1     &2   &2   &8.1603e$-$07   &0.0099             \\
			RECIPE  &3     &3        &2  &2     &3   &3   &3.0307e$-$15   &0.0118            \\
			S235    &3     &1        &16  &17    &17   &17  &1.8935e$-$08   &0.0230 \\
			S252    &3     &1     &14  &14    &15  &15   &4.7686e$-$08   &0.0267  \\
			S265    &4     &2     &1   &2     &2   &2   &1.8081e$-$16   &0.0073 \\
			S269    &5     &3     &5   &5     &6    &6  &2.1785e$-$07   &0.0201 \\
			S316    &2     &1     &2   &2     &3    &3  &4.8916e$-$08   &0.0094          \\
			S317    &2     &1     &5   &5     &6    &6  &3.3780e$-$12   &0.0143 \\
			S318    &2     &1     &5   &5     &6    &6  &3.2496e$-$11   &0.0143  \\
			S319    &2     &1     &7   &7     &8    &8  &6.3140e$-$08   &0.0151  \\
			S320    &2     &1     &25  &46    &26   &26  &9.1824e$-$07   &0.0498  \\
			S321    &2     &1     &17  &35    &18   &18  &2.9056e$-$07   &0.0406  \\
			S335    &3     &2     &11  &11    &12   &12  &1.2615e$-$07   &0.1254  \\
			S336    &3     &2     &7   &7     &8    &8  &2.3466e$-$08   &0.0211\\
			S338    &3     &2     &6   &6     &7    &7  &9.4251e$-$07   &0.0209  \\
			S344    &3     &1     &8   &10    &9    &9  &6.4195e$-$07   &0.0146 \\
			S373    &9     &6     &13  &12    &14   &14  &7.6678e$-$07   &0.1131 \\
			S378    &10    &3     &13  &13    &14   &14  &1.9417e$-$10   &0.0707 \\
			S394    &20    &12    &9   &9     &10   &10  &2.7574e$-$07   &0.1158  \\
			S395    &50    &1     &9   &8     &10   &10  &5.1312e$-$07   &0.2480  \\
			ZANGWIL3  &3    &3    &2   &2     &3    &3 &2.0128e$-$47   &0.0208 \\
			\bottomrule
		\end{longtable}
	}
\end{center}
\begin{center}
	{
		\begin{longtable}{cccccccccccccc}
			\caption{Comparison results 1}\label{table5.5}
			\endfirsthead
			\multicolumn{8}{l}{Table \ref{table5.5} continued}\\
			\hline
			\multirow{2}{*}{Problem}  & \multicolumn{2}{l}{Dimension}&&\multicolumn{2}{l}{Alg. 2.1 in \cite{Chen} }&&\multicolumn{2}{l}{LsFSARC}&&\multicolumn{2}{l}{LANCELOT} \\
			\cline{2-3}
			\cline{5-6}
            \cline{8-9}
            \cline{11-12}
			&$n$&$m$                       & &NF&NC& &NF                                   &NC        &&NF&NC&\\
			\hline
			\endhead
			\hline
			\multirow{2}{*}{Problem}  & \multicolumn{2}{l}{Dimension}& &\multicolumn{2}{l}{ Alg. 2.1 in \cite{Chen}}&&\multicolumn{2}{l}{LsFSARC}&&\multicolumn{2}{l}{LANCELOT} \\
			\cline{2-3}
			\cline{5-6}
            \cline{8-9}
            \cline{11-12}
			&$n$&$m$                       & &NF&NC& & NF                               & NC       &&NF&NC& \\
			\hline
			AIRCRFTA &8     &5            &&12   &21          &&3     &3   &&5  &5     \\
			ARGTRIG &200   &200          & &4   &4            &&3     &4    &&7   &7      \\
			BDVALUE  &102   &100          &&4   &6            &&3     &3    &&2 &2        \\
			BOOTH    &2     &2            &&2   &2            &&2     &2    &&3  &3 \\
			BROYDN3D &500   &500         &&14   &14          &&5     &5   &&6  &6  \\
			BT1     &2     &1            &&8   &9            &&5     &6    &&48 &48   \\
			BT2     &3     &1            &&11   &11          &&9     &10    &&22 &22       \\
			BT3     &5     &3            &&6  &6             &&4     &4      &&16 &16          \\
			BT4     &3     &2            &&9   &9            &&5     &1       &&28 &28           \\
			BT5     &3     &2            &&5   &5            &&8     &8   &&16 &16                 \\
			BT6     &5     &2            &&14   &14          &&15    &13    &&26 &26                  \\
			BT7     &5     &3            &&16  &18           &&4     &6   &&48&48    \\
			BT8     &5     &2            &&12   &19          &&6     &7    &&25 &25            \\
			BT9     &4     &2            &&13   &15          &&11    &10   &&20 &20          \\
			BT10    &2     &2            &&7   &7            &&2     &3    &&19 &19              \\
			BT11    &5     &3            &&8   &8            &&9     &10    &&19 &19                \\
			BT12    &5     &3            &&7   &8            &&8     &9      &&21 &21        \\
			BYRDSPHR &3     &2           &&10   &13           &&8     &9   &&22 &22                  \\
			CLUSTER  &2     &2            &&8   & 8           &&7     &5   &&10 &10           \\
			DTOC3    &299   &198         &&4   &4            &&28    &28    &&26  &26 \\
 			DTOC4    &299   &198         &&3   &  3          &&4     &4  &&17 &17\\
			DTOC5    &19   &9            &&4   &4            &&1     &1   &&14 &14   \\
			GENHS28  &10    &8            &&3   &3            &&5     &6  &&11 &11          \\
			GOTTFR  &2     &3             &&9   & 14          &&8     &6  &&17 &17       \\
			HAGER1   &1001  &500         &&2   &2            &&16    &17      &&13 &13  \\
			HAGER2   &1001  &500         &&6   &6            &&12    &13      &&12   &12  \\
			HAGER3   &1001  &500         &&14   &14          &&10    &11        &&13 &13\\ 			
            HATFLDF  &3     &3           &&28   &43          &&3     &4       &&62 &62\\
			HATFLDG  &25    &25           &&8   &9            &&2     &3   &&16 &16   \\ 			
			HEART8   &8     &8          &&48   &66          &&6     &6    &&149 &149  \\ 			
            HIMMELBA  &2     &2         &&3     &3          &&1     &2    &&3 &3  \\
			HIMMELBC  &2     &2         &&5     &5          &&2     &3    &&9 &9         \\
			HIMMELBE  &3     &3         &&3     &3          &&1     &2    &&6 &6           \\
			HS06    &2     &1           &&10   &13          &&13    &14     &&30 &30              \\
			HS07    &2     &1           &&9   &11           &&8     &8      &&17 &17                   \\
			HS08    &2     &2           &&5   &5            &&2     &3     &&10 &10            \\
			HS09    &2     &1           &&4   &4            &&7     &7     &&6 &6                \\
			HS26    &3     &1           &&19   &19          &&10    &10   &&22 &22       \\
			HS27    &3     &1           &&23   &25          &&29    &27   &&14 &14         \\
			HS28    &3     &1           &&2   &2            &&7     &6    &&7 &7   \\
			HS39    &4     &2           &&13   &15          &&11    &10   &&20 &20             \\
			HS40    &4     &3           &&4   & 4           &&19    &20   &&16 &16           \\
			HS42    &4     &2           &&4   & 4           &&52    &29   &&13 &13              \\
			HS46    &5     &2           &&17   & 17         &&13    &13   &&19 &19                     \\
			HS47    &5     &3           &&18   & 18         &&20    &21   &&21 &21                               \\
			HS48    &5     &2           &&3   &3            &&5     &5    &&8 &8                       \\
			HS49    &5     &2           &&15   &15          &&23    &23   &&18 &18                                    \\
			HS50    &5     &3           &&9   &9            &&15    &13   &&11 &11                             \\
			HS51    &5     &3           &&2   & 2           &&5     &4   &&7 &7                      \\
			HS52    &5     &3           &&3   & 3           &&6     &7   &&13 &13                     \\
			HS56    &7     &4           &&9   & 9           &&1     &1   &&18 &18                     \\
			HS61    &3     &2           &&6   & 6           &&6     &7   &&16 &16        \\
			HS77    &5     &2           &&16   &16          &&13    &12   &&23 &23             \\
			HS78    &5     &3           &&5  &5             &&14    &13   &&12 &12  \\
			HS79    &5     &3           &&5   & 5           &&8     &9    &&11 &11     \\
			HS100LNP  &7     &2         &&8    &10          &&21    &16   &&23 &23                   \\
			HS111LNP  &10    &3         &&12     &13        &&12    &13    &&69 &69        \\
			HYPCIR    &2     &2           &&5     &6          &&1     &2   &&9 &9                         \\
			INTEGREQ  &5     &5           &&2     & 2         &&1     &2   &&4 &4     \\
			MARATOS   &2     &1           &&4     &4          &&4     &4   &&9 &9                \\
			MWRIGHT   &5     &3           &&8     &8          &&10    &9   &&18 &18    \\
			ORTHREGB     &27         &6      &&25 &30             &&8     &8      &&74         &74           \\
			POWELLSQ  &2     &2           &&12    &15         &&1     &2   &&24 &24                 \\
			RECIPE  &3     &3             &&12   &12          &&2     &3    &&17 &17           \\
            ZANGWIL3  &3     &3              &&6     &6        &&2     &3   &&8 &8                     \\

			\bottomrule
		\end{longtable}
	}
\end{center}
\begin{center}
	{
		\begin{longtable}{ccccccccccccccccccccc}
			\caption{Comparison results 2}\label{table5.6}
			\endfirsthead
			\multicolumn{18}{l}{Table \ref{table5.6} continued}\\
			\hline
			\multirow{2}{*}{Problem}  & \multicolumn{2}{l}{Dimension}&&\multicolumn{3}{l}{Alg. 2.2 in \cite{liu2011SIOPT} }&&\multicolumn{3}{l}{LsFSARC}&&\multicolumn{3}{l}{LANCELOT} \\
			\cline{2-3}
			\cline{5-7}
            \cline{9-11}
            \cline{13-15}
			&$n$&$m$                       & &NF&NC&NG& &NF                                   &NC&NG        &&NF&NC&NG&\\
			\hline
			\endhead
			\hline
			\multirow{2}{*}{Problem}  & \multicolumn{2}{l}{Dimension}& &\multicolumn{3}{l}{ Alg. 2.2 in \cite{liu2011SIOPT}}&&\multicolumn{3}{l}{LsFSARC}&&\multicolumn{3}{l}{LANCELOT} \\
		    \cline{2-3}
			\cline{5-7}
           \cline{9-11}
            \cline{13-15}
			&$n$&$m$                       & &NF&NC&NG& & NF                               & NC&NG       &&NF&NC&NG& \\
			\hline
			AIRCRFTA &8     &5            &&3   &3&3          &&3     &3 &3  &&5  &5  &5   \\
			BDVALUE  &102   &100          &&3   &3&2            &&3     &3   &3 &&2 &2  &2      \\
			BOOTH    &2     &2            &&2   &2&2            &&2     &2   &2 &&4  &4 &4\\
			BT1     &2     &1            &&11   &11&8            &&5     &6  &6  &&57 &57 &47   \\
			BT3     &5     &3            &&8  &8 &8             &&4     &4   &4   &&15 &15 &15          \\
			BT4     &3     &2            &&14   &14 &14            &&5     &1 &6      &&27 &27 &26           \\
			BT5     &3     &2            &&9   &9 &9            &&8     &8  &8 &&67 &67   &43              \\
			BT6     &5     &2            &&30   &30 &29          &&15    &13  &13  &&51 &51 &39                 \\
			BT8     &5     &2            &&11   &11 &11          &&6     &7  &7  &&27 &27  & 25         \\
			BT9     &4     &2            &&57   &57 &41          &&11    &10  &10 &&23 &23 &23          \\
			BT10    &2     &2            &&8   &8 &8            &&2     &3  &3  &&21 &21 &21              \\
			BT11    &5     &3            &&13   &13 &13            &&9     &10  &10  &&23 &23 &20                \\
			BT12    &5     &3            &&9   &9&8            &&8     &9   &9   &&23 &23 &19       \\
			CLUSTER  &2     &2            &&8   & 8  &8         &&7     &5  &5 &&13 &13 &10           \\
			GENHS28  &10    &8            &&9   &9&9            &&5     &6 &6 &&10 &10 &10          \\
			GOTTFR  &2     &3             &&9   & 9&6          &&8     &6 &6 &&12 &12 &11       \\
			HATFLDG  &25    &25           &&25   &25 &7            &&2     &3 &3 &&24 &24 &20   \\ 			
            HIMMELBA  &2     &2         &&2     &2 &2          &&1     &2  &2  &&3 &3 &3  \\
			HIMMELBC  &2     &2         &&7     &7&6          &&2     &3  &3  &&9 &9 &8        \\
			HIMMELBE  &3     &3         &&3     &3  &3        &&1     &2  &6  &&4 &4 &4           \\
			HS06    &2     &1           &&14  &14 &11          &&13    &14  &14   &&58 &58 &42              \\
			HS07    &2     &1           &&12   &12 &12           &&8     &8   &8   &&24 &24 &19                   \\
			HS08    &2     &2           &&6   &6&5            &&2     &3   &3  &&11 &11  &10            \\
			HS09    &2     &1           &&7   &7 &7            &&7     &7   &7  &&11 &11 &11                \\
			HS26    &3     &1           &&36   &36 &26          &&10    &10  &10 &&33 &33 &31       \\
			HS28    &3     &1           &&10   &10 &9            &&7     &6 &6   &&4 &4 &4   \\
			HS39    &4     &2           &&57   &57 &41          &&11    &10  &10 &&674 &674 &630             \\
			HS40    &4     &3           &&7   & 7&7           &&19    &20  &20 &&15 &15 &14           \\
			HS42    &4     &2           &&11   & 11&9           &&52    &29  &29 &&13 &13 &13             \\
			HS46    &5     &2           &&29   & 29 &27         &&13    &13  &13 &&28 &28 &25                     \\
			HS48    &5     &2           &&13   &13 &10           &&5     &5  &5  &&4 &4 &4                       \\
			HS49    &5     &2           &&27   &27 &22          &&23    &23  &23 &&25 &25 &25                                    \\
			HS50    &5     &3           &&25   &25 &15            &&15    &13 &13  &&19 &19 &19                             \\
			HS51    &5     &3           &&10   & 10 &9           &&5     &4  &4 &&3 &3 &3                      \\
			HS52    &5     &3           &&8   & 8 &7           &&6     &7  &7 &&11 &11 &11                    \\
			HS61    &3     &2           &&13   &13 & 11          &&6     &7  &7 &&18 &18 &17        \\
			HS77    &5     &2           &&29   &29 &26          &&13    &12  &12 &&35 &35 &30             \\
			HS78    &5     &3           &&9  &9 &9             &&14    &13   &13&&26 &26 &15  \\
			HS79    &5     &3           &&13   & 13 &13           &&8     &9  &9  &&12 &12 &12     \\
			HS100LNP  &7     &2         &&79    &79 &35          &&21    &16 &16 &&510 &510 &468                     \\
			HYPCIR    &2     &2           &&6     &6  &5        &&1     &2  &2 &&7 &7&7                         \\
			INTEGREQ  &5     &5           &&2     & 2  &2       &&1     &2  &2 &&3 &3 &3     \\
			MARATOS   &2     &1           &&5     &5 &5          &&4     &4  &4 &&9 &9 &9               \\
			ORTHREGB  &27   &6         &&7  &7  &7             &&8     &8      &8        &&140 &140&116           \\
			RECIPE  &3     &3             &&12   &12 &12         &&2     &3  &3  &&43 &43   &37        \\
            ZANGWIL3  &3     &3              &&2     &2 &3       &&2     &3  &3 &&8 &8  &8                    \\

			\bottomrule
		\end{longtable}
	}
\end{center}
\begin{figure}[H]
  \centering
  \includegraphics[width=7cm,height=5cm]{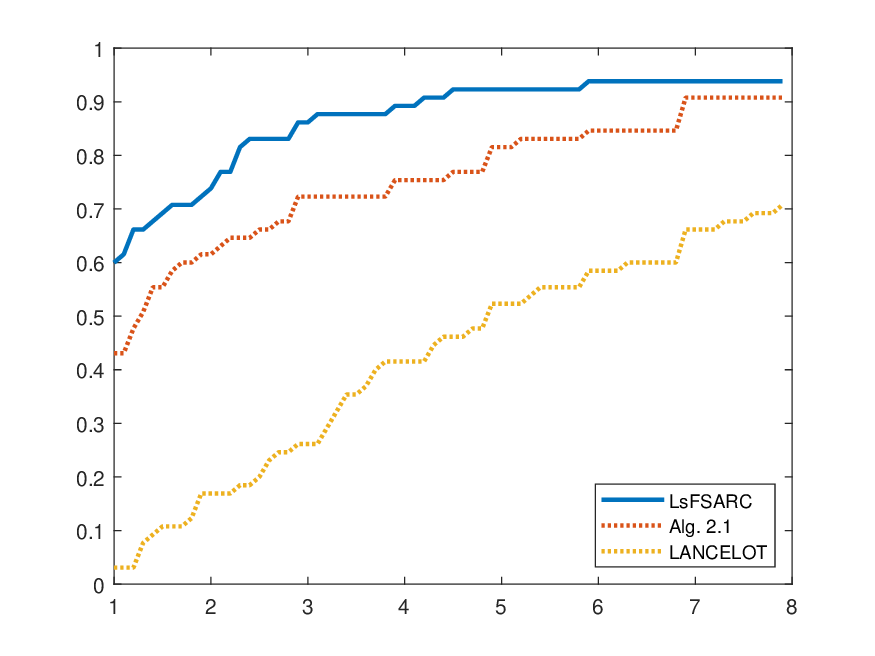}
  \quad
  \includegraphics[width=7cm,height=5cm]{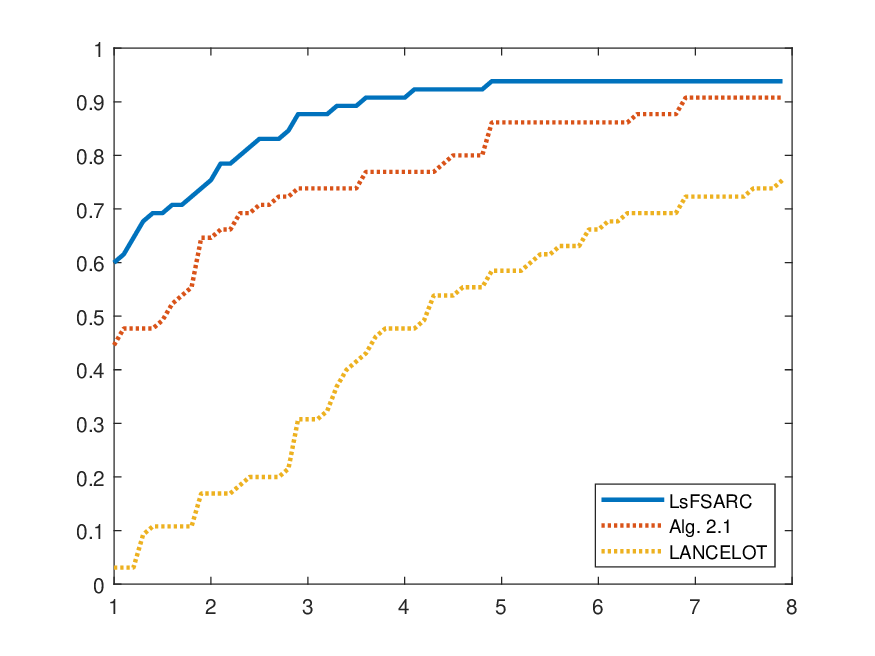}\\[0pt]
  \caption{ Performance profiles based on NF (left) and NC (right)}\label{Figure1}
\end{figure}

\begin{figure}[H]
  \centering
  \includegraphics[width=7cm,height=5cm]{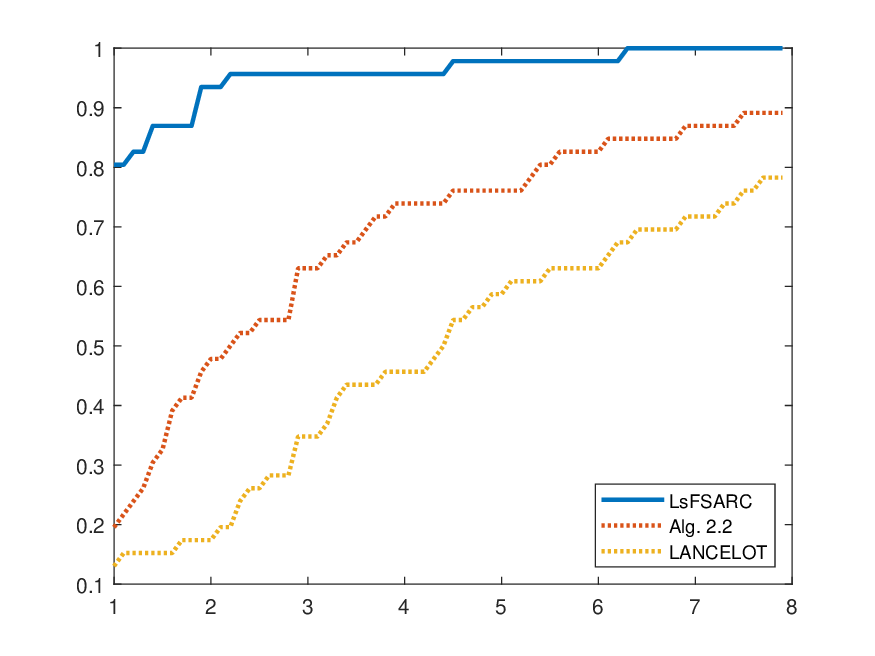}
  \quad
  \includegraphics[width=7cm,height=5cm]{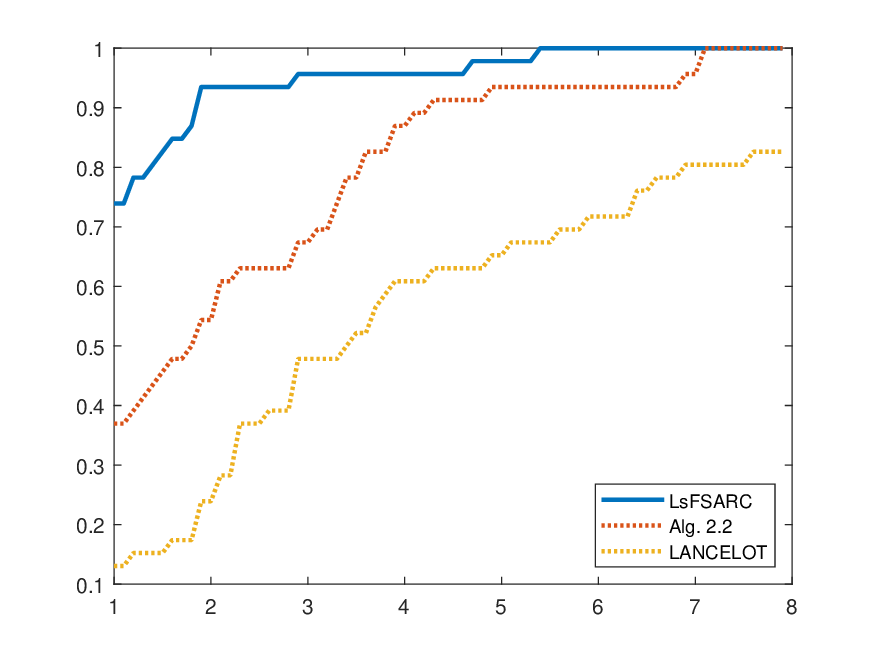}\\[0pt]
  \caption{ Performance profiles based on NF (left) and NG (right)}\label{Figure2}
\end{figure}

\section{Conclusion and discussion}\label{sec5}
We have introduced a line search filter sequential adaptive regularization algorithm using cubics (LsFSARC) to solve nonlinear equality constrained programming. It benefits from the idea of SQP methods. Composite step methods and  projective matrices are used to obtain the new step which is decomposed into the sum of a normal step and a tangential step. The tangential step is computed by solving a standard ARC subproblem. The global convergence analysis is reported under some suitable assumptions.  Preliminary numerical results and comparison results are presented to demonstrate the performance of LsFSARC. It can be observed that LsFSARC can be comparable with Algorithm 2.1 in \cite{Chen} and Algorithm 2.2 in \cite{liu2011SIOPT} for these test problems.

Compared with SQP algorithms where penalty function is employed as a merit function, LsFSARC  is a penalty-free method and does not involve the calculation and update of penalty parameters. Naturally, the convergence analysis does not rely on the penalty parameters. Moreover, compared with two impressive and powerful penalty-free methods in \cite{2005Line,liu2011SIOPT} which both require that the Lagrangian Hessian or its approximation is uniformly positive definite on the null space of the Jacobian of constraints for each $k$ to guarantee the descent property of the search directions, LsFSARC only requires semipositive definiteness of Lagrangian Hessian to ensure the global convergence.

However, the convergence analysis is not complete since local convergence properties are not discussed. The proposed algorithm LsFSARC can also suffer Maratos effect. To avoid this and achieve fast convergence rate, we can introduce second-order corrections or other techniques in the algorithm. Moreover, we are working on the worst-case complexity bound for the number of iterations to find an $\epsilon$-approximate KKT point.
\section*{Acknowledgments}
The authors are very grateful to the editor and the referees, whose valuable suggestions and insightful comments helped to improve the paper's prestentation significantly.  The authors gratefully acknowledge the supports of the National Natural Science Foundation of China (12071133).


\end{document}